\documentclass{article}

\usepackage{amssymb,amsfonts}
\usepackage[all,arc]{xy}
\usepackage{enumerate}
\usepackage{mathrsfs}
\usepackage[english]{babel}
\usepackage{amsthm}

\usepackage{latexsym}
\usepackage{amsmath}
\usepackage{verbatim}
\usepackage{dsfont}
\usepackage{titlesec}
\usepackage{tikz}
\usepackage{float}
\usepackage{pgfplots}
\usepackage{hyperref}
\usepackage{graphicx,caption} 

\newtheorem{thm}{Theorem}
\newtheorem{cor}[thm]{Corollary}
\newtheorem{prop}[thm]{Proposition}
\newtheorem{lem}[thm]{Lemma}

\theoremstyle{definition}
\newtheorem{defn}[thm]{Definition}
\newtheorem*{defn*}{Definition}
\newtheorem*{ex*}{Example}

\theoremstyle{remark}

\tikzset{
  treenode/.style = {shape=rectangle,align=center,draw,rounded corners},
  env/.style      = {treenode, font=\large},
  dummy/.style    = {rectangle, inner sep = 0.03in},
  coredummy/.style   = {rectangle},
  bdummy/.style    = {circle, inner sep = 0, outer sep = -0.5pt},
  tgraph/.style   = {circle, inner sep = 0pt,left=0.2in,draw,fill=black,outer sep = -0.5pt},
  adiag/.style = {circle, inner sep = 0pt,draw,fill=black,outer sep = -0.5pt},
  avertex/.style = {circle, inner sep = 0pt,draw,fill=black,minimum width = 3pt},
  circ/.style = {black,circle,draw,fill=white,inner sep = 1.5pt, outer sep = 0pt},
}

\tikzset{
  jumpdot/.style={mark=*,solid},
  excl/.append style={jumpdot,fill=white},
  incl/.append style={jumpdot,fill=black},
}

\usetikzlibrary{decorations.markings}

\newcommand{\C}{\mathbb{C}}
\newcommand{\D}{\mathcal{D}}

\renewcommand{\H}{\mathcal{H}}
\newcommand{\I}{\mathcal{I}}
\newcommand{\J}{\mathcal{J}}
\newcommand{\K}{\mathcal{K}}

\newcommand{\N}{\mathbb{N}}
\renewcommand{\P}{\mathscr{P}}

\newcommand{\Z}{\mathbb{Z}}

\newcommand{\abs}[1]{\lvert #1 \rvert}

\newcommand{\gen}[1]{\langle #1 \rangle}

\date{\today}
\author{Jordan Nikkel, Yunxiang Ren}
\title{On Jones' Subgroup of R. Thompson's Group $T$}

\pgfplotsset{compat=1.13}

\begin{document}
\maketitle

\begin{abstract}
Jones introduced unitary representations for the Thompson groups $F$ and $T$ from a given subfactor planar algebra.  Some interesting subgroups arise as the stabilizer of certain vector, in particular the Jones subgroups $\vec{F}$ and $\vec{T}$.  Golan and Sapir studied $\vec{F}$ and identified it as a copy of the Thompson group $F_3$. In this paper we completely describe $\vec{T}$ and show that $\vec{T}$ coincides with its commensurator in $T$, implying that the corresponding unitary representation is irreducible.  We also generalize the notion of the Stallings 2-core for diagram groups to $T$, showing that $\vec{T}$ and $T_3$ are not isomorphic, but as annular diagram groups they have very similar presentations. 
\end{abstract}

\tableofcontents

\section{Introduction}

The Thompson group $T$ is commonly defined as a particular subgroup of homeomorphisms of the unit circle \cite{cfp96}. Identify the unit circle with the interval $[0,1]$ modulo its boundary, and call a point on the unit circle dyadic if it is an integer multiple of an integer power of $2$. Then $T$ is the group of all homeomorphisms of the unit circle such that for each homeomorphism, there is a partition of the unit circle into finitely many intervals with dyadic endpoints, and on each interval that homeomorphism is linear with slope some integer power of $2$.  More generally, $T_n$ is the group of piecewise linear homeomorphisms of the unit circle, with each homomorphism having finitely many breakpoints at $n$-adic rational numbers and slopes integer powers of $n$ \cite{b87}.  $F$ and $F_n$ are the subgroups of $T$ and $T_n$ respectively that stabilize $0$.  These groups have been the subject of many studies, and recently that of Vaughan Jones \cite{j14,j16}.

Specifically, there is a reconstruction problem in subfactor theory: for every subfactor of finite index, can we construct a conformal field theory such that the standard invariant, or at least the quantum double of the subfactor can be recovered? After splitting the CFT into two chiral halves, one needs to construct the conformal net, which consists of von Neumann algebras $\mathcal{A}(\mathcal{I})$ on a Hilbert space $\mathcal{H}$ for every open interval $\mathcal{I}\subset S^1$, and a continuous projective unitary representation of $\text{Diff}(S^1)$ on $\mathcal{H}$ satisfying certain axioms. Motivated by this problem and thinking of the Thompson group $T$ as a discrete version of $\text{Diff}(S^1)$, Jones constructed many interesting representations for the Thompson group $F$ and $T$ \cite{j14} using the data from a given subfactor planar algebra $\mathscr{P}_\bullet=\{\mathscr{P}_{n,\pm}\}_{n\in\N}$, where each $\mathscr{P}_{n,\pm}$ is a finite-dimensional $C^*$-algebra, and a normalized vector $R\in \P_{2,+}$. For the specific representations of $F$ and $T$ corresponding to the Temperley-Lieb planar algebra $\mathcal{TL}(\sqrt{2})$ \cite{JonPA} with normalized vector $R$ as a multiple of the second Jones-Wenzl idempotent \cite{j83}, Jones defined $\vec{F}$ and $\vec{T}$ as the stabilizers of a vector in $\P_{1,+}$ called the vacuum vector. Furthermore, he proved that every link arises as the matrix coefficient with respect to the vacuum vector of some element of $F$ and $T$, and every oriented link arises as the matrix coefficient of some element of $\vec{F}$ and $\vec{T}$.

Jones showed how to test directly if an element $f$ of $F$ or $T$ is in $\vec{F}$ or $\vec{T}$ respectively by constructing the Thompson graph of $f$ and showing that $f$ is in $\vec{F}$ or $\vec{T}$ respectively if and only if $f$ has bipartite Thompson graph.  Mark Sapir and Gili Golan studied $\vec{F}$ in \cite{gs15a} using this characterization. In particular, they proved $\vec{F}$ is isomorphic to $F_3$ and showed that $\vec{F}$ was precisely the subgroup of $F$ of all elements that preserve the parity of the sums of digits of all dyadic rationals in $[0,1]$ as binary words. They also proved that $\vec{F}$ was its own commensurator in $F$, and observed that this implies that the corresponding representation of $F$ is irreducible.

In this paper we explore the subgroup $\vec{T}$, Jones' subgroup of $T$.  We present a finite set of generators for $\vec{T}$ in Theorem \ref{vecTgenerators}, describe the relationship of $\vec{T}$ to the stabilizer of the parity of the sums of digits of the dyadic rationals in Theorem \ref{dyadicparity}, and show in Corollary \ref{cor:commensurator} that $\vec{T}$ coincides with its commensurator in $T$, which implies that the corresponding unitary representation of $T$ is irreducible. We also provide an explicit finite presentation for $\vec{T}$ as an abstract group.  

In order to further investigate $\vec{T}$, we extend the notion of the Stallings 2-core developed in \cite{gs15b} and \cite{g16} to $T$. The Stallings 2-core is a construction for diagram group analogous to the Stallings foldings for free groups, where the foldings in the 2-core are two dimensional, rather than one dimensional.  Specifically, given a subgroup $H$ of a diagram group $G$, the Stallings 2-core of $H$ is a labeled directed $2$-complex that accepts certain elements of $G$ and rejects others.  The subset of elements it accepts is a subgroup that contains $H$, but is sometimes larger than $H$.  Thus the construction provides a partial solution to the problem of determining whether $H$ is equal to $G$ or not.  The 2-core construction also provides some insight into the structure of the group, which is used in \cite{gs15b} to find certain maximal subgroups of $F$, and which we use to prove that the 2-core of $\vec{T}$ is itself.  This in turn gives a presentation of $\vec{T}$ as an annular diagram group.  We also show that although $\vec{T}$ and $T_3$ have presentations as annular diagram groups that differ only by a Tietze transformation, $\vec{T}$ and $T_3$ are not isomorphic.

We first present in Section \ref{prelim} many basic preliminaries of $F$ and $T$, including the definitions using functions, pairs of trees, and annular diagram groups, as well as presenting a useful operation $\oplus$ and the normal forms for elements of $F$.  In Section \ref{vecT} we formally define $\vec{T}$ using Thompson graphs, prove that $\vec{F}$ and only one additional element generate all of $\vec{T}$, and demonstrate that $\vec{T}$ consists exactly of all functions in $T$ that either always stabilize or always switch the parity of all dyadic fractions in $[0,1)$ as binary numbers.  We also show that the commensurator of $\vec{T}$ is itself, give a finite presentation for the group, and show that any proper homomorphism of $\vec{T}$ factors through a homomorphism onto the infinite dihedral group.  In the final section, we extend the notion of the Stallings 2-core to $T$.

In this paper, we frequently refer to the generators and relations from \cite{cfp96}, but do multiplication of elements of $T$ from left to right, rather than right to left.  As a result, by choosing to use the inverses of the generators used in \cite{cfp96}, the same familiar relations hold true.

\textbf{Acknowledgements.} The authors would like to thank Mark Sapir for his support and comments on the text, as well as Gili Golan for her many helpful conversations. The authors would also like to thank Vaughan Jones for his direction in this paper.

\section{Preliminaries on \texorpdfstring{$T$}T}\label{prelim}

\subsection{Elements of \texorpdfstring{$T$}T as pairs of trees}\label{sec:Pair of trees}
\begin{figure}[ht]\centering
$ x_0(t) = \begin{cases} 2t & t \in [0,\frac14) \\
                        t + \frac14 & t \in [\frac14,\frac12) \\
                        \frac12t +\frac12 & t \in [\frac12,1)
        \end{cases}$
\hspace{0.2in}$ x_1(t) = \begin{cases}  t & t \in [0,\frac12) \\
                        2t - \frac12 & t \in [\frac12,\frac58) \\
                        t + \frac18 & t \in [\frac58, \frac34) \\
                        \frac12t+\frac12 & t \in [\frac34,1)
        \end{cases}$
$ c(t) \phantom{_1}= \begin{cases}  \frac12t + \frac12 & t \in [0,\frac12) \\
                        t + \frac14 & t \in [\frac12,\frac34) \\
                        2t - \frac32 & t \in [\frac34,1)
        \end{cases}$
\caption{The three generators of $T$.}\label{generatorsofT}
\end{figure}

$T$ is generated by functions $x_0$, $x_1$, and $c$ which are defined in Figure \ref{generatorsofT} \cite{cfp96}, and $F$ is the subgroup of $T$ generated by $x_0$ and $x_1$.
However, there is another natural definition of $T$ by representing its elements with pairs of binary trees.

A full binary tree is a rooted tree where every vertex has either $0$ or $2$ children. Given any two full binary trees $R$ and $S$ with the same number of leaves $m$ and a positive integer $n \leq m$, we associate $(R, S, n)$ with an element of $T$ in the following way. Associate the interval $[0,1)$ with the root of each tree, and inductively if the interval $[a,b)$ is associated with a vertex, associate its left child with $[a, \frac{a+b}2)$ and its right child with $[\frac{a+b}2,b)$. Since $R$ and $S$ have the same number of leaves, we can identify them pairwise by identifying the first leaf of $R$ with the $n$th leaf of $S$, and proceeding to identify the remaining leaves in a cyclical manner.  Since the intervals associated with the leaves of each tree form a partition of the unit circle, the function in $T$ corresponding to $(R, S, n)$ is the one that linearly maps the intervals of each leaf of $R$ to the intervals of their identified leaves of $S$.  If $n = 1$, then the function fixes $0$ and hence is also in $F$.

We will refer to the tree whose leaves partition the domain of the function as the input tree, and the tree whose leaves partition the image of the function as the output tree.  When they are depicted as a pair of trees with their leaves identified either directly or by labeling, we will call the diagram a tree diagram for the given element of $T$, and $(R, S, n)$ is called a pair of trees representation for the corresponding function.

It is well known that all elements of $T$, and hence all elements of $F$, arise from tree diagrams \cite{cfp96}.  For example, the tree diagrams with labeled leaves for $x_0$, $x_1$, and $c$ are depicted in Figure \ref{genT}, while a tree diagram for $x_0$ with leaves directly identified is depicted in Figure \ref{x0updown}.  The tree diagrams for elements of $T$ with leaves directly identified can be naturally drawn on an annulus, with the root of one tree on the outer edge of the annulus and the root of the other on the inner edge, which motivates the definition of $T$ and $\vec{T}$ using annular diagram groups (see Section \ref{Tannulardiagramgroup} for the definition of annular diagram groups).
\begin{figure}[ht] 
\centering
\begin{tikzpicture}
  [
    grow                    = down,
    level distance          = 0.4in,
    sibling distance        = 0.5in,
    parent anchor           = center,]
  \node [dummy] {}
    child { 
      child {node [dummy] {1}}
      child {node [dummy] {2}} 
     }
    child {node [dummy] {3}};
\end{tikzpicture}
\begin{tikzpicture}
\draw (0,0) coordinate (b)
      (0,1.2) coordinate (a);
\draw[->] (a) arc (120:90:0.2in) node[above] {$x_0$} arc(90:60:0.2in);
\end{tikzpicture}
\begin{tikzpicture}
  [
    grow                    = down,
    sibling distance        = 0.5in,
    level distance          = 0.4in,
    parent anchor           = center,]
    \node [dummy] {}
      child {node [dummy] {1}}
      child { 
        child {node [dummy] {2}}
        child {node [dummy] {3}}
      };

\end{tikzpicture}
\hspace{0.15in}
\begin{tikzpicture}
  [
    grow                    = down,
    sibling distance        = 0.5in,
    level distance          = 0.4in,
    parent anchor           = center,
    baseline                = -0.875in,
  ]
  \node [dummy] {}
      child {node [dummy] {1}}
      child {
        child {
          child {node [dummy] {2}}
          child {node [dummy] {3}}
         }
        child {node [dummy] {4}}
      };
\end{tikzpicture}
\begin{tikzpicture}
\draw (0,0) coordinate (b)
      (0,1.2) coordinate (a);
\draw[->] (a) arc (120:90:0.2in) node[above] {$x_1$} arc(90:60:0.2in);
\end{tikzpicture}
\hspace{0.1in}
\begin{tikzpicture}
  [
    grow                    = down,
    sibling distance        = 0.5in,
    level distance          = 0.4in,
    parent anchor           = center,
    baseline                = -0.875in,
  ]
    \node [dummy] {}
    child {node [dummy] {1}}
    child { 
      child {node [dummy] {2}}
      child { 
          child {node [dummy] {3}}
          child {node [dummy] {4}}
        }
      };
\end{tikzpicture}
\end{figure}

\begin{figure}[ht] 
\centering
\begin{tikzpicture}
  [
    grow                    = down,
    level distance          = 0.4in,
    sibling distance        = 0.5in,
    parent anchor           = center,]
  \node [dummy] {}
    child { node [dummy] {1}}
    child { 
      child { node [dummy] {2}}
      child { node [dummy] {3}}
      };
\end{tikzpicture}
\begin{tikzpicture}
\draw (0,0) coordinate (b)
      (0,1.2) coordinate (a);
\draw[->] (a) arc (120:90:0.2in) node[above] {$c$} arc(90:60:0.2in);
\end{tikzpicture}
\hspace{.1in}
\begin{tikzpicture}
  [
    grow                    = down,
    sibling distance        = 0.5in,
    level distance          = 0.4in,
    parent anchor           = center,]
    \node [dummy] {}
    child { node [dummy] {3}}
    child { 
      child { node [dummy] {1}}
      child { node [dummy] {2}}
    };
\end{tikzpicture}\captionsetup{width=.9\linewidth}
\caption{The standard generators of $T$, where the numbers signify which leaves are identified in each pair of trees.}\label{genT}
\end{figure}
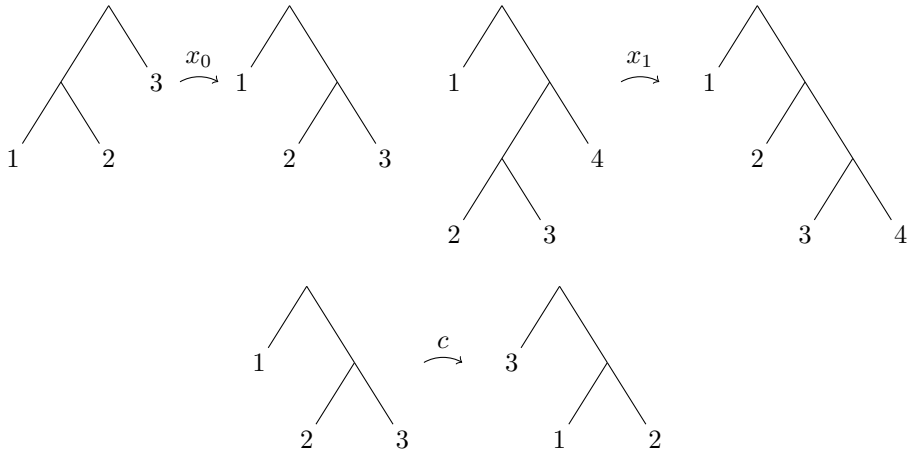

Sometimes it is convenient to not number the vertices, but simply label in the output tree with a circle which leaf is identified with the first leaf of the input tree, as is done in Figure \ref{dipole}.

By adding a caret, i.e. two children, to the same numbered leaf in both pairs of trees, and identifying the left and right children accordingly while fixing all other leaf identifications, the corresponding function is unchanged.  For example, Figure \ref{dipole} depicts $x_0$ with a caret added at the second vertex.

\begin{figure}[ht] 
\centering
\begin{minipage}{.4\textwidth}
\centering
\begin{tikzpicture}  [
    grow                    = down,
    level distance          = 0.3in,
    sibling distance        = 0.4in,
    parent anchor           = center,
    ]
    \node [adiag] (root) {}
    child { node[adiag]{}
      child {node[adiag]{}}
      child {node[adiag]{}}
     }
    child { 
      child {edge from parent [draw=none]}
      child {node[adiag]{}}
    };
    \node [adiag,below=1.2in](broot) at (root) {} [grow'=up]
    child {
      child 
      child {edge from parent [draw=none]}
    }
    child { node[adiag]{}
      child 
      child 
    };
\end{tikzpicture}\caption{Tree diagram for $x_0$ with corresonponding leaves identified directly.}\label{x0updown}
\end{minipage}
\hspace{0.05\textwidth}
\begin{minipage}{.5\textwidth}
\vspace{-0.19in}
\centering
\begin{tikzpicture}
  [
    grow                    = down,
    level distance          = 0.4in,
    sibling distance        = 0.5in,
    parent anchor           = center,]
  \node [dummy] {}
    child { 
      child {}
      child {
        child{}
        child{}
      } 
    }
    child {};
\end{tikzpicture}
\hspace{0.05in}
\begin{tikzpicture}
\draw (0,0) coordinate (b)
      (0,1.8) coordinate (a);
\draw[->] (a) arc (120:90:0.2in) node[above] {$x_0$} arc(90:60:0.2in);
\end{tikzpicture}
\hspace{0.05in}
\begin{tikzpicture}
  [
    grow                    = down,
    sibling distance        = 0.5in,
    level distance          = 0.4in,
    parent anchor           = center,]
    \node [dummy] {}
    child {node [circ] {}}
    child { 
      child {
        child{}
        child{}
      }
      child {}
    };
\end{tikzpicture}
\caption{A non-reduced representation of $x_0$, using the $\circ$ to denote which leaf of the output tree is identified with the first leaf of the input tree.}\label{dipole}
\end{minipage}
\end{figure}

Removing such pairs of carets that do not affect the corresponding function is called removing a dipole.  A tree diagram from which no dipole can be removed is called reduced, and every function has a unique reduced tree diagram \cite{cfp96}.

Since multiplication in $T$ is done via composition of functions, the same is represented with trees by doing the following: given two pairs of trees representations $(R_1, S_1, n_1)$ and $(R_2, S_2, n_2)$, add dipoles to both pairs of trees until $S_1$ and $R_2$ are the same tree $S$.  Then $(R_1, S, n_1) \cdot (S, S_2, n_2) = (R_1, S_2, n)$ where $n = n_1 + n_2 - 1$ and is taken modulo the number of leaves in $R_1$ and $S_2$.  Another way to view multiplication is to draw tree diagrams with leaves identified directly, as in Figure \ref{x0updown}, and then to identify the bottom vertex of the first diagram with the top vertex of the second.  If any vertex has children both above and below it, remove that vertex and identify its top and bottom left children, and its top and bottom right children.  This process may also be referred to as removing dipoles, and when all dipoles have been removed, the diagram is a standard tree diagram.

\subsection{\texorpdfstring{$T$}T as an annular diagram group} \label{Tannulardiagramgroup}
In \cite{gs97}, Guba and Sapir extend the notion of a diagram group to annular diagram groups.  Just as a diagram group can be defined as the fundamental group of a Squier complex, annular diagram groups can be defined as the fundamental group of a particular extension of a Squier complex.  Here we will focus on the diagrammatic definition.  

In short, let $\P = \gen{X | R}$ be a semi-group presentation, and $u$ a word over $X$, then the annular diagram group $\mathcal{D}^a(\P,u)$ is the group whose elements are finite and oriented graphs with the following properties.  The edges are labeled by elements of $X$, each diagram must have an inner vertex, which is on a distinguished inner path whose label is the same as $u$, and an outer vertex with distinguished outer path whose label is also $u$.  Finally, the inner and outer paths are tessellated by cells, which have inner and outer paths, where if an outer path of a cell is labeled by $r$ and the inner path is labeled by $s$, then $r = s$ or $s = r$ is a relation in $R$.  See Figure \ref{adiagex} for a sketch of a typical cell and annular diagram.  

If $r = s$ is the relation, the cell is called positive, and if $s = r$ is the relation, it is negative.  If the inner path of a cell $\Delta_1$ is labeled by $r$ with outer path labeled $s$, and the inner path of the cell is the outer path of another cell $\Delta_2$ labeled by $r$ with inner path labeled $s$, then this is called a dipole, and the two cells may be removed and the outer path of $\Delta_1$ is identified with the inner path of $\Delta_2$.  Any diagram without dipoles is called reduced, and multiplication in the group is done by identifying the inner path of one diagram with the outer path of another diagram and then reducing. Pictorally, this looks like put one annulus, like that in Figure \ref{adiagex}, inside another to create a third annulus.  

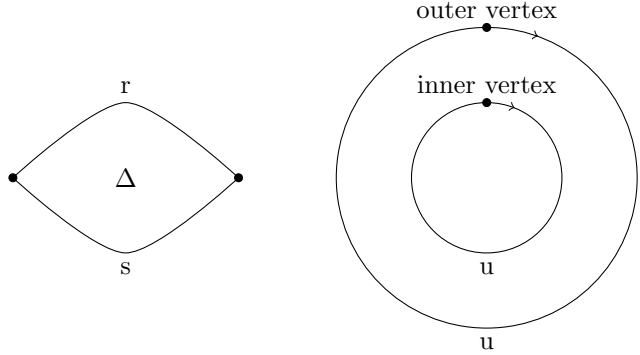
\begin{figure}[ht]\centering
\begin{tikzpicture}[baseline]
    \draw plot[smooth] coordinates{(0,0) (1.5,1) (3,0)};
    \draw plot[smooth] coordinates{(0,0) (1.5,-1) (3,0)};
    \node[above] at (1.5,1) {r};
    \node[below] at (1.5,-1) {s};
    \node at (1.5,0) {$\Delta$};
    \node[avertex] at (0,0) {};
    \node[avertex] at (3,0) {};
\end{tikzpicture}
\hspace{1cm}
\begin{tikzpicture}[baseline]
    \draw[
        decoration={markings, mark=at position 0.2 with {\arrow{<}}},
        postaction={decorate}
        ]
        (0,0) circle (2);
    \draw[
        decoration={markings, mark=at position 0.2 with {\arrow{<}}},
        postaction={decorate}
        ]
        (0,0) circle (1);
    \node[avertex] at (0,1) {};
    \node[above] at (0,1) {inner vertex};
    \node[avertex] at (0,2) {};
    \node[above] at (0,2) {outer vertex};
    \node[below] at (0,-1) {u};
    \node[below] at (0,-2) {u};
\end{tikzpicture}\captionsetup{width=.8\linewidth}
\caption{On left is a typical cell $\Delta$ in an annular diagram, where the top and bottom paths are labeled by $r$ and $s$, which are words over the $X$, the generating set, and $r = s$ is a relation in $R$.  The right picture is a sketch of an annular diagram, with the inner and outer vertices labeled, and inner and outer paths depicted with orientation and labeled by $u$, which is likewise a word over $X$.  In between these two paths would be cells like $\Delta$, which tesselate the inside of the annulus.}\label{adiagex}
\end{figure}

Guba and Sapir discovered that $T$ is the annular diagram group $\mathcal{D}^a(\gen{x | x = x^2}, x)$ \cite{gs97}.  In order to draw annular diagrams for an element of $T$ with pair of trees representation $(R,S,n)$, it is convenient to draw separately the top half, which corresponds to $R$, and the bottom half, which corresponds to $S$, and then to glue them together according to $n$. For example, Figure \ref{adiagramc} is an annular diagram for $c$, one of the generators of $T$, with the same $\circ$ notation used in tree diagrams for identifying vertices.  Since all edges are labeled by $x$, the labels are omitted.  The figure also shows the relationship between annular diagrams of elements of $T$ and pairs of trees.  Specifically, every cell, each of which has a path with one edge and a path with two edges, can be associated with a caret by putting a vertex in the middle of each edge of the cell, and connecting the vertex on the path with one edge to both of the other vertices.

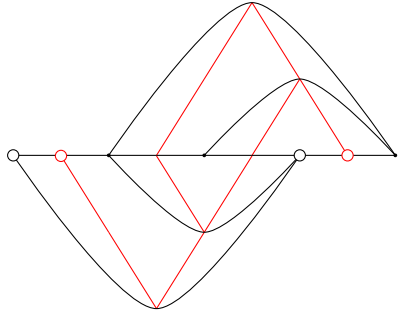
\begin{figure}[ht]
\centering

\begin{tikzpicture}  [
    grow                    = down,
    level distance          = 0.4in,
    sibling distance        = 0.5in,
    parent anchor           = center,
    ]
  \node [dummy](root) {}
    child[red] { {node [bdummy](t15){}}
      child {node [bdummy](t1){}}
      child {node [bdummy](t2){} edge from parent [draw=none]}
      }
    child[red] { {node [bdummy](t35){}}
      child { node [bdummy](t3){}}
      child { node [bdummy](t4){}}
    };
    \node[adiag, left = 0.25in](v0) at (t1) {};
    \node[adiag, right = 0.5in](v1) at (v0) {};
    \node[adiag, right = 1in](v2) at (v0) {};
    \node[adiag, right = 1.5in](v3) at (v0) {};
    \node[adiag, right = -0.5in](v00) at (v0) {};
    \draw (v00)--(v3);
    \draw plot[smooth] coordinates{(v1) (t35) (v3)};
    \draw plot[smooth] coordinates{(v0) (root) (v3)};
    \node [dummy,below=0.77in](broot) at (t1) {} [grow'=up]
        child[red] { {node [bdummy](bt15){}}
          child {node [bdummy](bt1){}}
          child {node [bdummy](bt2){} edge from parent [draw=none]}
          }
        child[red] { {node [dummy](bt35){}}
          child { node [bdummy](bt3){}}
          child { node [bdummy](bt4){}}
        };
    \draw plot[smooth] coordinates{(v00) (broot) (v2)};
    \draw plot[smooth] coordinates{(v0) (bt35) (v2)};
    \node [circ,red, fill=white] at (t4){};
    \node [circ,red, fill=white] at (bt1){};
    \node [circ] at (v00){};
    \node [circ] at (v2){};
\end{tikzpicture}\captionsetup{width=.81\linewidth}
\caption{The annular diagram for $c$, which is in black and where vertices labeled with $\circ$ are identified.  The pair of trees for $c$ is also depicted in red, where the leaves labeled with $\color{red}{\circ}$ are identified.  Here, the top most arch, colored black, is the outer path of cell, and the bottom most arch is the inner path of the cell.}\label{adiagramc}
\end{figure}


\subsection{Addition and normal forms of elements of \texorpdfstring{$F$}F}\label{sec:normalforms}
A useful operation on $F$ is the following addition operator \cite{gs15a}.  If $f, g \in F$, then $f\oplus g$ is defined in the following way:

\[(f\oplus g)(t) = \begin{cases}
    \frac{f(2t)}2 & t \in [0,\frac12] \\
    \frac{g(2t-1)+1}2 & t \in (\frac12,1] \\
    \end{cases} \]
    
If $f = (R_1, S_1)$ and $g = (R_2, S_2)$, then similarly $f\oplus g = (R,S)$ where $R$ is a tree where the left child of the root is a copy of $R_1$ and a right copy of the root is $R_2$, and $S$ is likewise defined.  Notice that although $F$ is closed under addition, $T$ is not.  In fact, if $g \in T \setminus F$, then since $g$ does not stabilize $0$, $g \oplus f$ for any $f \in T$ will not be a continuous function.

With this notation, notice that $x_1 = 1 \oplus x_0$.  In general, another set of standard generators of $F$ are given by $\{x_n\}_{n \in \N}$ where $x_n = 1 \oplus x_{n-1}$, and $x_0$ is as already defined.  These generators are useful in that every element of $F$ be written as a product of $x_0^{a_0}x_1^{a_1}\ldots x_n^{a_n}x_n^{-b_n}\ldots x_1^{-b_1}x_0^{-b_0}$ where $n$, $a_i$, and $b_i$ are all non-negative integers \cite{cfp96}.  Then $x_0^{a_0}x_1^{a_1}\ldots x_n^{a_n}$ is called the positive part of $f$, and $x_n^{-b_n}\ldots x_1^{-b_1}x_0^{-b_0}$ is called the negative part.  If the negative part of $f$ is empty, then $f$ is said to be a positive element, and likewise if the positive part is empty, then $f$ is a negative element.

Furthermore, given a pair of trees representation of $f = (R,S)$, $n$, $a_i$, and $b_i$ can all be determined in the following way \cite{cfp96}: define the exponent of a vertex of a tree to be the maximal length of a path of left edges that end at the vertex and begins not on the right most path of the tree.  Then $n+1$ is the number of leaves in $R$ and $S$, and $a_i$ is the exponent of the $i$th vertex of $R$ numbered from $0$ to $n$, and $b_i$ is the exponent of the $i$th vertex of $S$ numbered similarly. 

This naturally leads to the following observation about the structure of reduced tree diagrams for elements of $F$.  Consider the tree $S_n$ in Figure \ref{treeSn}.  If the output tree in a tree diagram for an element of $F$ is $S_n$ for some $n$, then the exponents $b_i$ are all $0$, and hence the element is positive.  Likewise, if the input tree is $S_n$, then the exponents $a_i$ are all $0$, and hence the element is negative.

\begin{figure}[ht]
\centering
\begin{tikzpicture}
  [
    grow            = down,
    sibling distance= 0.4in,
    level distance  = 0.3in,
    parent anchor   = center,
    sloped
  ]
    \node [dummy](root) {}
    child{node [dummy] {1}}
    child{
        child{node[dummy] {2}}
        child{node [dummy] {}
            child{node[dummy] {n+1}}
            child{node[dummy] {n+2}}
            edge from parent [draw=none] node {$\ldots$} 
        }
    };
\end{tikzpicture}\caption{The tree $S_n$.}\label{treeSn}
\end{figure}
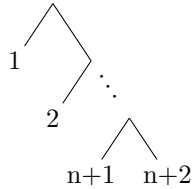

\section{Jones' subgroup \texorpdfstring{$\vec{T}$}{of T} and its properties}\label{vecT}
\subsection{Thompson graphs and \texorpdfstring{$\vec{T}$}{Jones' subgroup of T}}

Given any full binary tree with $n$ leaves, the associated Thompson graph is defined by Jones \cite{j14} in the following way.  Arrange all the leaves on a horizontal line, and call them $l_1,\ldots,l_n$.  The vertices of the Thompson graph are points $v_1,\ldots,v_n$ on the same horizontal line, with $v_1$ to the left of $l_1$, and more generally $v_i$ between $l_{i-1}$ and $l_i$.  For every left edge $e$ of the tree, there is a unique pair of vertices $v_i$ and $v_j$ such that a path can be drawn connecting $v_i$ and $v_j$ which passes through the tree only at the edge $e$ and stays above the horizontal line.  Connect every such pair of vertices. An example Thompson graph is depicted in Figure \ref{tgraphex}.

\begin{figure}[ht]
\centering
\begin{tikzpicture}
  [
    grow            = down,
    sibling distance= 0.7in,
    level distance  = 0.5in,
    parent anchor   = center,
    baseline        = 0in,
    level 1/.style  ={sibling distance=0.8in},
    level 2/.style  ={sibling distance=0.4in},
  ]
    \node [dummy](root) {}
    child {{node [dummy](l1) {}}
        child{node [dummy](1) {}}
        child{node [dummy](2) {}}
    }
    child{
        child[level distance=0.25in,sibling distance = 0.2in]{
            child{node [dummy](3) {}}
            child{node [dummy](4) {}}
        }
        child{node [dummy](5) {}}
    };
    \node[dummy,above=0.25in](tl1) at (l1) {};
    \node[dummy,above=0.25in](tl2) at (1) {};
    \node[dummy,above=0.25in,left=0.02in](tl3) at (3) {};
    \node[dummy,above=0.15in](tl4) at (4) {};
    \node[dummy,above=0.6in,right=0.02in](tl5) at (3) {};
    \node[tgraph](t1) at (1) {};
    \node[tgraph](t2) at (2) {};
    \node[tgraph](t3) at (3) {};
    \node[tgraph,left=-0.1in](t4) at (4) {};
    \node[tgraph,left=-0.1in](t5) at (5) {};
    \draw [red](t1).. controls (tl1)..(t3);
    \draw [red](t1).. controls (tl2)..(t2);
    \draw [red](t3).. controls (tl3)..(t4);
    \draw [red](t3).. controls (tl5)..(t5);
\end{tikzpicture}\captionsetup{width=.7\linewidth}
\caption{Example of a Thompson graph in red of a binary tree in black.}\label{tgraphex}
\end{figure}
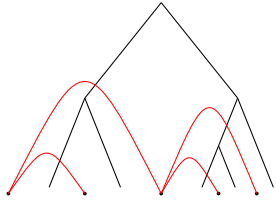

The Thompson graph associated to any element $f \in T$ with reduced pair of trees representation $(R, S, n)$ is obtained by identifying the vertices of the Thompson graphs of $R$ and $S$ in the same way that their leaves are identified. An example is depicted in Figure \ref{tgraphex2}, with the Thompson graph both depicted on the pair of trees and then simplified.

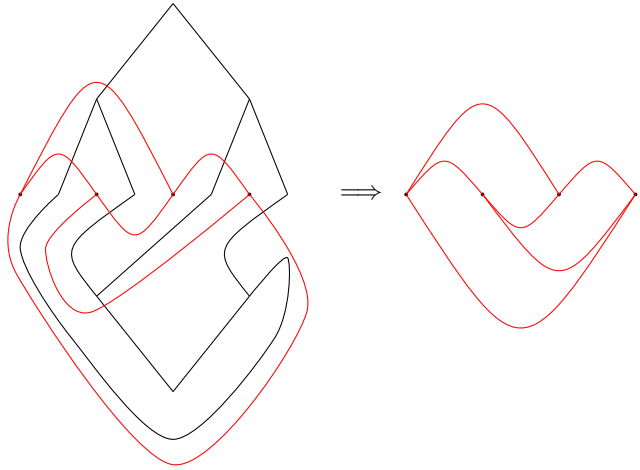
\begin{figure}[ht]
\centering
\begin{tikzpicture}
  [
    grow            = down,
    sibling distance= 0.7in,
    level distance  = 0.5in,
    parent anchor   = center,
    baseline        = 0in,
    level 1/.style  ={sibling distance=0.8in},
    level 2/.style  ={sibling distance=0.4in},
  ]
    \node [dummy](root) {}
    child {{node [dummy](l1) {}}
        child{node [bdummy](1) {}}
        child{node [bdummy](2) {}}
    }
    child{
        child{node [bdummy](3) {}}
        child{node [bdummy](4) {}}
    };
    \node[dummy,above=0.25in](tl1) at (l1) {};
    \node[dummy,above=0.25in](tl2) at (1) {};
    \node[dummy,above=0.25in](tl3) at (3) {};
    \node[tgraph](t1) at (1) {};
    \node[tgraph](t2) at (2) {};
    \node[tgraph](t3) at (3) {};
    \node[tgraph](t4) at (4) {};
    \draw [red](t1).. controls (tl1)..(t3);
    \draw [red](t1).. controls (tl2)..(t2);
    \draw [red](t3).. controls (tl3)..(t4);
    \node [dummy,below=2in] at (root) {} [grow'=up]
        child{node [bdummy](bl1) {}}
        child{node [bdummy](bl2) {}};
    \node[dummy,below=0.25in](b1) at (1) {};
    \node[dummy,below=0.25in](b2) at (2) {};
    \node[dummy,below=0.25in](bt1) at (t1) {};
    \node[dummy,below=0.25in](b3) at (3) {};
    \node[dummy,below=0.3in](b4) at (4) {};
    \node[dummy,below=2.25in](broot) at (root) {};
    \node[dummy,below=1.75in,right=0.5in](brootr) at (root) {};
    \node[dummy,below=1.75in,left=0.5in](brootl) at (root) {};
    \node[dummy,right=0.1in](c1) at (bt1) {}; 
    \node[dummy,right=0.1in,above=0.1in](c2) at (brootl) {};
    \node[dummy,left=0.02in,below=0.1in](d1) at (bt1) {};
    \node[dummy,below=0.1in](d2) at (broot) {};
    \node[dummy,right=0.1in,below=0.25in](d3) at (b4) {};
    \draw (bl1.center).. controls (b1).. (2);
    \draw (bl1.center) to (3);
    \draw (bl2.center).. controls (b3)..(4);
    \draw plot[smooth] coordinates{(bl2.center) (b4) (brootr) (broot) (brootl) (bt1) (1.south west)};
    \draw [red](t2).. controls (b2)..(t3);
    \draw [red] plot[smooth] coordinates{(t2.south west) (c1) (c2) (t4)};
    \draw [red] plot[smooth] coordinates{(t1.south west) (d1) (d2) (d3) (t4.south east)};
\end{tikzpicture}
\begin{tikzpicture}[baseline=1in]
      \node[dummy] {$\implies$};
\end{tikzpicture}
\begin{tikzpicture}[baseline=1in]
    \node[tgraph](1) {};
    \node[tgraph,right=0.6in](2) at (1){};
    \node[tgraph,right=0.6in](3) at (2){};
    \node[tgraph,right=0.6in](4) at (3){};
    \node[dummy,right=0.2in,above=0.2in](12) at (1){};
    \node[dummy,above=0.6in](13) at (2){};
    \node[dummy,right=0.2in,above=0.2in](34) at (3){};
    \node[dummy,right=0.2in,below=0.9in](b14) at (2){};
    \node[dummy,below=0.5in](b24) at (3){};
    \node[dummy,right=0.2in,below=0.2in](b23) at (2) {};
    \draw [red] (1).. controls (12).. (2);
    \draw [red] (1).. controls (13).. (3);
    \draw [red] (3).. controls (34).. (4);
    \draw [red] (1).. controls (b14).. (4);
    \draw [red] (2).. controls (b24).. (4);
    \draw [red] (2).. controls (b23).. (3);
\end{tikzpicture}\captionsetup{width=.8\linewidth}
\caption{Thompson graph in red of an element of $T$ depicted both on the pair of trees diagram and simplified.}\label{tgraphex2}
\end{figure}

The subgroup $\vec{T}$ of $T$ then is the collection of all elements of $T$ with bipartite Thompson graphs, and was first defined by Jones when he also defined $\vec{F}$ similarly.  We will use the terms bipartite and 2-colorable interchangably. For more about $\vec{F}$, see \cite{gs15a}, in which Golan and Sapir found explicit generators of $\vec{F}$ and discovered many other properties and characterizations of the subgroup.

\subsection{The generators of \texorpdfstring{$\vec{T}$}{Jones' subgroup of T}}\label{sec:genofT}
Denote by $c_n$ the element depicted in Figure \ref{cn}.

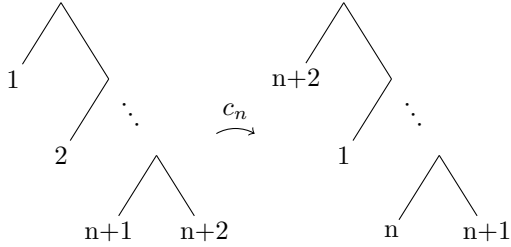
\begin{figure}[ht]
\centering
\begin{tikzpicture}
  [
    grow            = down,
    sibling distance= 0.5in,
    level distance  = 0.4in,
    parent anchor   = center,
    sloped
  ]
    \node [dummy](root) {}
    child{node [dummy] {1}}
    child{
        child{node[dummy] {2}}
        child{node [dummy] {}
            child{node[dummy] {n+1}}
            child{node[dummy] {n+2}}
            edge from parent [draw=none] node {$\ldots$} 
        }
    };
\end{tikzpicture}
\hspace{-0.2in}
\begin{tikzpicture}
\draw (0,0) coordinate (b)
      (0,1.5) coordinate (a);
\draw[->] (a) arc (120:90:0.2in) node[above] {$c_n$} arc(90:60:0.2in);
\end{tikzpicture}
\begin{tikzpicture}
  [
    grow            = down,
    sibling distance= 0.5in,
    level distance  = 0.4in,
    parent anchor   = center,
    sloped
  ]
    \node [dummy](root) {}
    child{node [dummy] {n+2}}
    child{
        child{node[dummy] {1}}
        child{node [dummy] {}
            child{node[dummy] {n}}
            child{node[dummy] {n+1}}
            edge from parent [draw=none] node {$\ldots$} 
        }
    };
\end{tikzpicture}\caption{The element $c_n$ from $T$.}\label{cn}
\end{figure}
Notice that $c = c_1$ is one of the standard generators of $T$, and $c_n$ in general is an order $n+2$ element.  Let $f_{\frac12}$ denote the function given by $f_\frac12(x) = x + \frac12 \mod 1$.  Its pair of trees diagram is depicted in Figure \ref{f12}, and $f_{\frac12}^2 = 1$.  

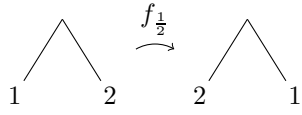
\begin{figure}[ht]
\centering
\begin{tikzpicture}
  [
    grow            = down,
    sibling distance= 0.5in,
    level distance  = 0.4in,
    parent anchor   = center,
    sloped
  ]
    \node [dummy](root) {}
    child{node [dummy] {1}}
    child{node[dummy] {2}
    };
\end{tikzpicture}
\begin{tikzpicture}
\draw (0,0) coordinate (b)
      (0,0.8) coordinate (a);
\draw[->] (a) arc (120:90:0.2in) node[above] {$f_\frac12$} arc(90:60:0.2in);
\end{tikzpicture}
\begin{tikzpicture}
  [
    grow            = down,
    sibling distance= 0.5in,
    level distance  = 0.4in,
    parent anchor   = center,
    sloped
  ]
    \node [dummy](root) {}
    child{node [dummy] {2}}
    child{node[dummy] {1}
    };
\end{tikzpicture}\caption{The element $f_{\frac12}$.}\label{f12}
\end{figure}

By Theorem 1 of \cite{gs15a}, we have $\vec{F} = \gen{x_0x_1,x_1x_2,x_2x_3}$, where $x_i$ are the generators defined in Section \ref{sec:normalforms}.  The goal of this section is the following theorem, which adds just one element to the set of generators of $\vec{F}$ to get the generators of $\vec{T}$.

\begin{thm}\label{vecTgenerators}
$\vec{T} = \gen{\vec{F},f_\frac12} = \gen{x_0x_1,x_1x_2,x_2x_3,f_\frac12}$
\end{thm}

Before we prove the theorem, we need the following lemma to establish the structure of the tree diagrams of elements of $\vec{T}$.

\begin{lem}\label{lemalternating}
Every element $f \in \vec{T}$ can be represented with a pair of trees ($R$, $S$, $k$) such that both trees have an even number of vertices, and such that there is a 2-coloring on the Thompson graphs of $R$ and $S$ where the colors on each vertex from left to right alternate.\end{lem}
\begin{proof}
Let $(R,S,k)$ be the reduced pair of trees representation for $f \in \vec{T}$.  By definition of $\vec{T}$, the Thompson graph of $f$ has a valid 2-coloring, which induces a 2-coloring on the vertices of the Thompson graph of $R$.  Suppose that the Thompson graph of $f$ contains two adjacent vertices $m$ and $m+1$ with the same color, and consider inserting a dipole at vertex $m+1$, i.e., adding a caret to the $m+1$ vertices of both $R$ and $S$.  Adding a caret adds one left edge and right edge to $R$, and hence the corresponding Thompson graph of $R$ is changed by adding a vertex in the middle of the caret, and connecting that vertex through the new left edge to the existing vertex on its left. A likewise addition is made to the Thompson graph of $S$, so that the new Thompson graphs of both $R$ and $S$ are obtained by adding a vertex in between vertices $m$ and $m+1$, with an edge connecting the new vertex to $m$.  Since $m$ and $m+1$ have the same color, we may choose the other color for the new vertex, and the $2$-coloring is still valid.

Continue adding dipoles in $R$ and $S$ in this way to add vertices to the Thompson graphs of $R$ and $S$ between all two vertices with the same color, ensuring that the colors alternate as desired.  Likewise add a dipole at the last vertex of $R$ and $S$ if necessary to ensure that $R$ and $S$ has an even number of vertices.
\end{proof}

\begin{proof}[Proof of Theorem \ref{vecTgenerators}]
Let $\vec{F}_+$ denote the positive elements of $F$ that are in $\vec{F}$ and $\vec{F}_-$ denote the negative elements of $F$ that are in $\vec{F}$, where positive and negative elements are as defined in Section \ref{sec:normalforms}.  Let $S_n$ be the tree depicted in Figure \ref{treeSn}, and note that $S_n$ is both the input and output tree for $c_n$ in Figure \ref{cn}.  Then to prove the theorem we can use the previous lemma to first prove that elements of $\vec{T}$ can be written in the form $pc_{2n}^mq$, where $p \in \vec{F}_+$ and $q \in \vec{F}_-$, which is very similar to the form of elements for $T$ given in Theorem 5.7 of \cite{cfp96}.  Thus we will have proven that $\vec{T} = \gen{\vec{F}_+,\vec{F}_-,\{f_\frac12\},\{c_{2n} | n \in \N\}}$, and to finish the proof it will suffice to show that each of these generators is in $\gen{x_0x_1,x_1x_2,x_2x_3,f_\frac12}$.

Let $f \in \vec{T}$, then by Lemma \ref{lemalternating}, there exists $(R,S,k)$, a pair of trees representation of $f$, such that the Thompson graphs for $R$ and $S$ have $2n+2$ vertices and are both 2-colorable with the colors on the vertices alternating.  Consider the reduced pair of trees representation for $c_{2n}$ denoted $(S_{2n},S_{2n},2)$.  It is easy to observe that the Thompson graph corresponding to $S_{2n}$ is simply the path of length $2n+2$, and hence it is 2-colorable with the colors on the vertices alternating, just like the Thompson graphs of $R$ and $S$.

In particular, $(R,S_{2n},1)$, $(S_{2n},S_{2n},k)$, and $(S_{2n},S,1)$ are all bipartite elements of $T$ since the 2-colorings of all of the trees are compatible.  Furthermore, by the observation given in section \ref{sec:normalforms}, $(R,S_{2n},1)$ is in $F_+$ and $(S_{2n},S,1)$ is in $F_-$. We also have $(S_{2n},S_{2n},k) = c_{2n}^{k-1}$.  Therefore $f \in \gen{\vec{F}_+,\vec{F}_-,\{c_{2n} | n \in \N\}}$, which is a subgroup of $\vec{T}$ since all generating elements have bipartite Thompson graphs and are in $T$.

Finally, since $\vec{F}$ is generated by $\{x_0x_1,x_1x_2,x_2x_3\}$, it suffices to show that $c_{2n} \in \gen{\vec{F},f_{\frac12}}$.  Observe that $f_\frac12c_n^{-1} \in F$:

\begin{tikzpicture}
  [
    grow            = down,
    sibling distance= 0.5in,
    level distance  = 0.4in,
    parent anchor   = center,
    sloped
  ]
    \node [dummy](root) {}
    child{
        child{node [dummy] {1}}
        child{
            child{node[dummy] {2}}
            child{node [dummy] {}
                child{node[dummy] {n}}
                child{node[dummy] {n+1}}
                edge from parent [draw=none] node {$\ldots$} 
            }
        }
    }
    child{node[dummy]{n+2}};
\end{tikzpicture}
\hspace{-0.2in}
\begin{tikzpicture}
\draw (0,0) coordinate (b)
      (0,2) coordinate (a);
\draw[->] (a) arc (120:90:0.2in) node[above] {$f_\frac12$} arc(90:60:0.2in);
\end{tikzpicture}
\begin{tikzpicture}
  [
    grow            = down,
    sibling distance= 0.5in,
    level distance  = 0.4in,
    parent anchor   = center,
    sloped
  ]
    \node [dummy](root) {}
    child{node [dummy] {n+2}}
    child{
        child{node [dummy] {1}}
        child{
            child{node[dummy] {2}}
            child{node [dummy] {}
                child{node[dummy] {n}}
                child{node[dummy] {n+1}}
                edge from parent [draw=none] node {$\ldots$} 
            }
        }
    };
\end{tikzpicture}
\hspace{-0.2in}
\begin{tikzpicture}
\draw (0,0) coordinate (b)
      (0,2) coordinate (a);
\draw[->] (a) arc (120:90:0.2in) node[above] {$c^{-1}_n$} arc(90:60:0.2in);
\end{tikzpicture}
\begin{tikzpicture}
  [
    grow            = down,
    sibling distance= 0.5in,
    level distance  = 0.4in,
    parent anchor   = center,
    sloped
  ]
    \node [dummy](root) {}
    child{node [dummy] {1}}
    child{
        child{node [dummy] {2}}
        child{
            child{node[dummy] {3}}
            child{node [dummy] {}
                child{node[dummy] {n+1}}
                child{node[dummy] {n+2}}
                edge from parent [draw=none] node {$\ldots$} 
            }
        }
    };

\end{tikzpicture}

Therefore, since $f_\frac12$ and $c_{2n}$ are bipartite, $f = f_\frac12 c_{2n}^{-1}$ is a bipartite element of $\vec{F}$.  In particular, $c_{2n} = f_\frac12 f^{-1}$, so $c_{2n} \in \gen{\vec{F},f_\frac12}$.  
\end{proof}

\subsection{\texorpdfstring{$\vec{T}$}{Jones' subgroup of T} and dyadic parity}
Every dyadic fraction has either an even or odd sum of binary digits, and we will refer to this parity as the dyadic parity of the number.  Observe that since every function $f \in T$ is piecewise of the form $ax + b$, where $a$ is an integer power of $2$ and $b$ is dyadic, if $f$ stabilizes odd dyadic parity, it also stabilizes even dyadic parity.

Since the generators of $\vec{F}$ preserve dyadic parity \cite{gs15a}, and since $f_\frac12$ switches dyadic parity, the subgroup $\vec{T}$ that they generate has the property that every element either exclusively preserves or switches dyadic parity.  In fact, this property characterizes $\vec{T}$.

\begin{thm}\label{dyadicparity}
Let $f \in T$ either stabilize or switch dyadic parity.  Then $f \in \vec{T}$.
\end{thm}
\begin{proof}
Let $f$ have reduced pair of trees representation $(R,S,k)$.  Label left edges of $R$ and $S$ with $0$ and right edges with $1$, and label each leaf with the label of the path from the root of the tree to that leaf.  Then each vertex $v$ in the Thompson graph of $R$ is to the left of some leaf $a$ in $R$. If $a$ has even dyadic parity then color the corresponding vertex in the Thompson graph ``even", and otherwise color it ``odd". Likewise the vertices in the Thompson graph of $S$ can be colored.

If $a_1$ and $a_2$ are the labels of leaves in $R$ and $S$ respectively that are identified, then $f(a_1) = a_2$.  Therefore, if $f$ stabilizes the dyadic parity then $a_1$ and $a_2$ have the same parity. and if $f$ switches the dyadic parity, then they have opposite parity.  Thus, when the Thompson graphs of $R$ and $S$ are identified, up to swapping the colors, the colorings are identical.

It remains to show that for the Thompson graphs of $R$ and $S$, this is indeed a valid 2-coloring, i.e., that no two adjacent vertices have the same color.  Suppose two vertices $v_1$ and $v_2$ in the Thompson graph of a $R$ are adjacent, and the leaves to their right are labeled by $a_1$ and $a_2$.  Then either one of them is directly beneath a caret, or neither is.  Both situations are depicted in Figure \ref{tgraphadjacentvertices}, and when one is directly beneath a caret, it is easily verified that $a_1$ and $a_2$ differ only by the last digit, and hence $v_1$ and $v_2$ have opposite colorings.  In second case, the edge between $v_1$ and $v_2$ crosses a left edge $e$ in $R$, which is labeled $0$, and let $k$ be the length of the path from the root to $e$, including $e$. Now the vertex in $R$ labeled by $a_1$ is to the left of this edge in $R$, hence the $k$th digit of $a_1$ is $0$, while the $k$th digit of $a_2$ is labeled $1$.  Furthermore, until the $k$th digits, $a_1$ and $a_2$ have the same digits.  Finally, note that after the digits past the $k$th digits of $a_1$ and $a_2$ are all $0$s, as that is the only way for these two vertices to be on either side of $e$.  Indeed, if the path to either vertex ever went right, then the edge between the vertices in the Thompson graph would have to cross more than just the edge $e$ of $R$. Thus the dyadic parity of $a_1$ and $a_2$ are exactly opposite, and the $2$-coloring is valid.
\end{proof}

\begin{figure}[ht]\centering
\begin{tikzpicture}
  [
    grow            = down,
    sibling distance= 0.5in,
    level distance  = 0.4in,
    parent anchor   = center,
    level 2/.style  ={level distance = 0.15in},
    level 3/.style  ={sibling distance=1.2in,level distance = 0.4in},
    level 4/.style  ={sibling distance=0.4in},
  ]
    \node [dummy](root) {}
    child{edge from parent [draw=none] node[sloped] {$\ldots$} child {edge from parent [draw=none]
        child{node[dummy](e1){}
            child{
                child{node[dummy](a1){$a_1$} edge from parent node[left](a13){$0$}}
                child{node[dummy](a3){$a_3$} edge from parent node[right]{$1$}}
                edge from parent [draw=none] node[sloped] {$\ldots$}
            }
            child{edge from parent[draw=none]}
            edge from parent node[left,above]{$0$} node[right]{$e$} 
        }
        child{
            child{
                child{node[dummy](a2){$a_2$} edge from parent node[left]{$0$}}
                child{edge from parent [draw=none]}
                edge from parent [draw=none] node[sloped] {$\ldots$}
            }
            child{edge from parent[draw=none]}
            edge from parent node[right,above]{$1$}
        }}
    };
    \node[tgraph,red](v1) at (a1) {};
    \node[tgraph,red](v2) at (a2) {};
    \node[tgraph,red](v3) at (a3) {};
    \node[dummy,below,red] at (v1) {$v_1$};
    \node[dummy,below,red] at (v2) {$v_2$};
    \node[dummy,below,red] at (v3) {$v_3$};
    \node[dummy,above=0.5in](e) at (e1) {};
    \draw [red] (v1).. controls (a13).. (v3);
    \draw [red] (v1).. controls (e).. (v2);
\end{tikzpicture}
\captionsetup{width=.84\linewidth}
\caption{The two possible types of adjacencies in a Thompson graph, between $v_1$ and $v_2$ and $v_1$ and $v_3$.  The leaves labeled $a_1$ and $a_2$ have different dyadic parity, as do the leaves labeled $a_1$ and $a_3$.}\label{tgraphadjacentvertices}
\end{figure}
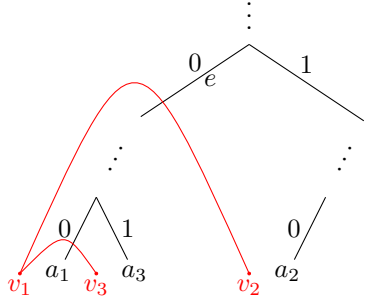

\subsection{\texorpdfstring{$\vec{T}$}{Jones' subgroup of T} coincides with its commensurator in \texorpdfstring{$T$}{T}}
The following Lemma formalizes the statement that for any element $f \in T$, for $t$ small, $f$ either preservers the dyadic parity of all such $t$, or switches the dyadic parity.  The proof simply relies on the fact that elements of $T$ are piecewise linear with slopes equal to an integer power of $2$.

\begin{lem}
Let $f \in T$, and let $S_i$ denote the collection of dyadic rational numbers with sum $i \mod 2$.  Then there exists $m$ such that for all $t \in S_i$ with $t\leq 2^{-m}$ and $i = 0,1$, either $f(t)$ is always in $S_i$ or it is always in $S_{1-i}$.
\end{lem}
\begin{proof}
First, there exists $m$ such that on $[0,2^{-m}]$, $f$ is given by $f(t) = 2^nt + \alpha$ for some $n \in \Z$ and $\alpha$ a dyadic rational.  Let $t \in S_i$.  Consider the length of $\alpha$ to be $l$, then $m$ be chosen to be large enough that for $t \leq 2^{-m}$, $t\cdot2^n$ can be written as a binary decimal of the form $0^l\beta$ for some finite binary string $\beta$ depending on $t$, where $\beta \in S_i$, the same as $t$.  Then $f(t) = 2^nt + \alpha$ as a binary string is simply $\alpha\beta$, hence if $\alpha \in S_0$, then $f(t)$ does not change the dyadic parity on $[0,2^{-m}]$, and if $\alpha \in S_1$, then $f(t)$ does switch dyadic parity on $[0,2^{-m}]$.
\end{proof}

With this technical fact, we can now prove that $\vec{T}$ coincides with its commensurator in $T$ as a corollary of Theorem \ref{dyadicparity}, where the commensurator of $\vec{T}$ is defined as $\{t \in T \:|\: t\vec{T}t^{-1}\cap \vec{T} \text{ has finite index in both } \vec{T} \text{ and } t\vec{T}t^{-1}\}$.

\begin{cor}\label{cor:commensurator}The commensurator of $\vec{T}$ in $T$ is $\vec{T}$.
\end{cor}
\begin{proof}
Let $f \in T \setminus \vec{T}$.  Then to check that $f$ is not in the commensurator of $\vec{T}$, it suffices to show that there exists $g \in \vec{T}$ such that for large enough $n$, $(g^n)^f \not\in \vec{T}$.  Choose $g = (x_0x_1)^{-1}$, and note that for any $t \in (0,1)$, $g^n(t)$ can be made arbitrarily small by taking $n$ large enough. 

Since $f \not\in \vec{T}$, there exists $t_0 \in S_i$ and $t_1 \in S_j$ for some $i, j \in \{0,1\}$ such that $f(t_0) \in S_i$ and $f(t_1) \not\in S_j$.  Let $t_k' = f(t_k)$ for $k = 1, 2$.  Then consider $(f^{-1}g^nf)(t_k') = f(g^n(f^{-1}(t_k'))) = f(g^n(t_k))$.  For large enough $n$, $g^n(t_i)$ is small enough that Lemma 1 applies to $f$.  Now $t_1' \in S_i$ and $t_2' \not\in S_j$ by assumption, noting that elements of $T$ preserve dyadic rationals, so $t_2' \in S_{1-j}$.  Since $g \in \vec{F}$, $g^n(t_1) \in S_i$ and $g^n(t_2) \in S_j$.  

By the previous lemma there are two cases. Suppose that for small enough values of $t$, $f$ preserves dyadic parity.  Then $f(g^n(f^{-1}(t_1'))) \in S_i$ with $t_1' \in S_i$, and $f(g^n(f^{-1}(t_2'))) \in S_j$ with $t_2' \not\in S_j$, showing that $f^{-1}g^nf \not\in \vec{T}$.  The other case where $f$ switches the parity is similar.
\end{proof}

 It was proven in \cite{} that the quasiregular representation of a subgroup is irreducible if the commensurator of the subgroup coincides with itself. 
 Thus from Corollary \ref{cor:commensurator}, we have the following theorem:
\begin{thm}
  The quasiregular representation $\ell^2(T/\vec{T})$ of $T$ is irreducible.   
\end{thm}

\subsection{A presentation of \texorpdfstring{$\vec{T}$}{Jones' subgroup of T}}\label{sec:presentation}

In this section, we determine first an infinite classical presentation for $\vec{T}$, and then deduce a finite presentation. Our infinite set of generators consists of all $x_{n-1}x_n = g_n$ and $c_{2n}$, for $n \in \Z$ and $n\geq 0$, using $c_0 = f_{\frac12}$ for convenience.  Note that each $g_n$ is indeed in $\vec{F}$ by Lemma 4.5 of \cite{gs15a}.

In the proof of Lemma \ref{lem:basic relation}, we will use the following relations of $T$ from \cite{cfp96} which hold for any integers $n,k$ such that $0\leq k \leq n$:

\begin{align}
    x_k^{-1}x_nx_k&=x_{n+1},~~~~~~k<n;\tag{T1}\label{T1}\\
    c_n&=x_nc_{n+1};\tag{T2}\label{T2}\\
    c_nx_k&=x_{k-1}c_{n+1},~~~~~~1\leq k;\tag{T3}\label{T3}\\
    c_nx_0&=c_{n+1}^2;\tag{T4}\label{T4}\\
    c_n^{n+2}&=1;\tag{T5}\label{T5}
\end{align}

\begin{lem}\label{lem:basic relation}
    If $n$ is a non-negative integer then
    \begin{align}
	g_k^{-1}g_ng_k&=g_{n+2},~~~1\leq k<n;\label{rel0a}\\
	c_{2n}^{2n+2}&=1;\label{rel0b}\\
    c_{2n}&=g_{2n+1}c_{2n+2};\label{rel1}\\
    c_{2n}g_k&=g_{k-1}c_{2n+2};~~~1<k<2n + 2\label{rel2}\\
    c_{2n}g_1&=c_{2n+2}^3\label{rel3}
    \end{align}
\end{lem}
\begin{proof}
\begin{align*}
g_k^{-1}g_ng_{k} &= (x_k^{-1}x_{k-1}^{-1})(x_{n-1}x_n)(x_{k-1}x_{k})\\ 
&= x_k^{-1}x_{k-1}^{-1}x_{n-1}(x_{k-1}x_{k-1}^{-1})x_nx_{k-1}x_{k}\\
&\overset{\ref{T1}}= x_k^{-1}x_nx_{n+1}x_{k}\\
&= x_k^{-1}x_n(x_kx_k^{-1})x_{n+1}x_{k}\\
&\overset{\ref{T1}}= x_{n+1}x_{n+2} \\
&= g_{n+2}
\end{align*}

Thus (\ref{rel0a}) holds.  (\ref{rel0b}) is the same as (\ref{T5}).

\[c_{2n} \overset{\ref{T2}}= x_{2n}c_{2n+1} \overset{\ref{T2}}= x_{2n}x_{2n+1}c_{2n+2} = g_{2n+1}c_{2n+2}\]

Hence (\ref{rel1}) holds.  Relations (\ref{rel2}) and (\ref{rel3}) are proven similarly:

\begin{align*}
c_{2n}g_k & = c_{2n}x_{k-1}x_k \\
&\overset{\ref{T3}}=x_{k-2}c_{2n+1}x_{k}\\
&\overset{\ref{T3}}=x_{k-2}x_{k-1}c_{2n+2}\\
&= g_{k-1}c_{2n+2}
\end{align*}
\vspace{-1em}
\begin{align*}
c_{2n}g_1 &= c_{2n}x_0x_1\\
&\overset{\ref{T4}}= c_{2n+1}^2x_1\\
&\overset{\ref{T3}}= c_{2n+1}x_0c_{2n+2}\\
&\overset{\ref{T4}}= c_{2n+2}^3
\end{align*}
\end{proof}

These relations also determine directly the following relations.
\begin{cor}\label{cor:extrarel}
    Suppose $n$ is a non-negative integer and $1\leq m\leq 2n+1$, then
    \begin{align}
    c_{2n}^m&=g_{2n+1-(m-1)}c_{2n+2}^m \label{rel4}\\
    c_{2n}^m&=c_{2n+2}^{m+2}g_m^{-1} \label{rel5}
    \end{align}
\end{cor}
\begin{proof}
The first relation is proven in the following way:

\[c_{2n}^m = c_{2n}^{m-1}c_{2n} \overset{(\ref{rel1})}= c_{2n}^{m-1}g_{2n+1}c_{2n+2} \overset{(\ref{rel2})}= g_{2n+1-(m-1)}c_{2n+2}^m\]

For the second relation, we instead prove that $c_{2n}^mg_m = c_{2n+2}^{m+2}$. 

\[c_{2n}^mg_m \overset{(\ref{rel2})}= c_{2n}g_1c_{2n+2}^{m-1} \overset{(\ref{rel3})}= c_{2n+2}^{m+2}\]
\end{proof}

\begin{cor}\label{cor:movegkleft}
Let $k$, $n$, and $m$ be non-negative integers such that $k < 2n + 2$ and $1\leq m < 2n+2$.  Then
    \[c_{2n}^{m}g_k = \begin{cases}  g_{k-m}c_{2n+2}^m & k > m \\
    c_{2n+2}^{m+2} & k = m\\
    g_{2n+2+k-m}c_{2n+2}^{m+2} & k < m
    \end{cases}\]
\end{cor}
\begin{proof}
If $k > m$, repeatedly apply (\ref{rel2}) to get the first relation.  If $k = m$, the relation is the same as (\ref{rel5}).  If $k < m$, then we use the same technique to get \[c_{2n}^mg_k = c_{2n}^{m-k}c_{2n+2}^{k+2} \overset{(\ref{rel4})}= g_{2n+1 - (m-k-1)}c_{2n+2}^{m-k+k+2} = g_{2n+2+k-m}c_{2n+2}^{m+2}\]
\end{proof}

\begin{cor}\label{cor:absrotation}
    Suppose $i,j,k,l$ are positive integers and $i<2j + 2, k<2l + 2$, then there exists $m,n$
    are positive integers with $m < 2n + 2$, and $p, q$ are positive elements in $\{g_t,t\geq1\}$ such that
    \begin{equation*}
    c_{2j}^i c_{2l}^k=p c_{2n}^m q^{-1}
    \end{equation*}
\end{cor}
\begin{proof}
By using relations (\ref{rel4}) and (\ref{rel5}), we can increase $j$ or $i$ until $i=j$, up to multiplying by a postive element on the left or a negative element on the right.  Since $c_{2n}$ is an order $2n+2$ element, we may assume that $0 \leq m < 2n+2$.
\end{proof}

Now we determine an infinite presentation for the subgroup $\vec{T}$.

\begin{thm}\label{thm:presentation}
    $\vec{T}$ is generated by $g_k,k\geq1$ and $c_n,n\geq0$ and is defined by the following relations:
    \begin{align*}
    g_k^{-1}g_ng_k&=g_{n+2}~~~~~~1\leq k<n\\ 
    c_{2n}^{2n+2}&=1\\
    c_{2n}&=g_{2n+1}c_{2n+2}\\
    c_{2n}g_k&=g_{k-1}c_{2n+2}~~~~~~1<k<2n+2\\
    c_{2n}g_1&=c_{2n+2}^3
    \end{align*}
\end{thm}

Before we proceed to the proof, we need some lemmas to establish the structure that these relations provide.  For the rest of this section, we will refer to $G$ as the group with presentation given in Theorem \ref{thm:presentation}, and show that $G$ is isomorphic to $\vec{T}$.

\begin{lem}\label{lem:vecFsubgroup}
$\vec{F}$ is a subgroup of $G$.
\end{lem}
\begin{proof}
Since the relations of $G$ also hold in $\vec{T}$, there is a natural homomorphism $\phi$ from $G$ to $\vec{T}$, sending generators to generators.  Since $\vec{F}$ has a finite presentation \cite{gs97}, and its relations are included in the relations of $G$ (using $g_k$ as its generators), there is also a homomorphism $\alpha$ from $\vec{F}$ to $G$ that takes generators to generators.  Thus $\alpha\circ\phi$ is a homomorphism taking $g_k$ to $g_k$, and hence is the identity isomorphism of $\vec{F}$, implying that $\alpha$ is also an isomorphism, and that $\vec{F}$ is a subgroup of $G$.
\end{proof}

\begin{lem}\label{lem: standard form}
  For every $g\in G$, we have 
    \begin{equation*}
    g=pc_{2n}^mq^{-1},~~0 \leq m < 2n+2,
    \end{equation*}
    where $p,q$ are positive elements in $\vec{F} = \{g_k,k\geq1\} \leq G$.
\end{lem}
\begin{proof}
    Let $H$ be the subset of $G$ consisting of all elements of the form $pc_{2n}^mq^{-1}$, with $m < 2n+2$, and $p$ and $q$ positive elements in $\vec{F}$.  Since $H$ contains the generators of $G$ and is closed under taking inverses, it suffices to prove that $H$ is closed under multiplication, and hence $H = G$.  Let $p_1c_{2n_1}^{m_1}q_1^{-1}$ and $p_2c_{2n_2}^{m_2}q_2^{-1}$ be two elements of $H$, and note that the product of two positive elements of $F$ is still an element of $F$ \cite{cfp96}, hence the same is true in $\vec{F}$.  Moreover, $p_1c_{2n_1}^{m_1}q_1^{-1}p_2c_{2n_2}^{m_2}q_2^{-1} = p_1c_{2n_1}^{m_1}p_3q_3^{-1}c_{2n_2}^{m_2}q_2^{-1}$ for some positive elements $p_3$ and $q_3$, since $q_1^{-1}p_2$ is in $\vec{F}$.  
    
    First we show that $p_1c_{2n_1}^{m_1}p_3$ can be written as $p_4c_{2n_3}^{m_3}$ for some positive element $p_4$ and $n_3, m_3$ natural numbers. Let $g_k$ to be the first letter in $p_3$.  By Corollary \ref{cor:movegkleft}, $g_k$ can be moved to the left of $c_{2n_1}^{m_1}$, with possibly increasing $n_1$ or $m_1$.  Continue in this way to move all of $p_3$ to the left of $c_{2n_1}^{m_1}$, to get the desired form $p_4c_{2n_3}^{m_3}$, since the product of positive elements of $\vec{F}$ is still  a positive element.
    
    Similarly, since $(q_3^{-1}c_{2n_2}^{m_2}q_2^{-1})^{-1} = q_2c_{2n_2}^{-m_2}q_3$, we can rewrite $q_3^{-1}c_{2n_2}^{m_2}q_2^{-1}$ as $c_{2n_4}^{m_4}q_4^{-1}$.  Finally, $c_{2n_3}^{m_3}c_{2n_4}^{m_4}$ can be written as $p_5c_{2n}^mq_5^{-1}$ by Corollary \ref{cor:absrotation}.  Thus 
    \begin{align*}
    p_1c_{2n_1}^{m_1}q_1^{-1}p_2c_{2n_2}^{m_2}q_2^{-1} &=p_1c_{2n_1}^{m_1}p_3q_3^{-1}c_{2n_2}^{m_2}q_2^{-1} = p_4c_{2n_3}^{m_3}c_{2n_4}^{m_4}q_4^{-1} \\
    &= p_4p_5c_{2n}^{m}q_4^{-1}q_5^{-1} = pc_{2n}^mq^{-1}
      \end{align*}
    \end{proof}
    
    \begin{lem}\label{lem:properquotient}
    If $\phi: G \to G/N$ is a proper quotient homomorphism, then $\phi$ restricted to $\vec{F}$ is a proper quotient of $\vec{F}$ as well.
    \end{lem}
    \begin{proof}
     Let $g \in G$ such that $\phi(g) = 1$ and $g \not=1$.  It follows from Lemma \ref{lem: standard form} that $g$ is of the form $g = pc_{2n}^mq^{-1}$, where $m < 2n+2$ and $p$ and $q$ are positive elements in $\vec{F} \leq G$. Since $\phi(g)=1$, we have $\phi(c_{2n}^m)=\phi(p^{-1}q)$, and since the order of $c_{2n}$ is $2n+2$, then $\phi((p^{-1}q)^{2n+2}) = \phi(c_{2n}^{2n+2}) = \phi(1) = 1$.  We now have two cases: either $p = q$ or $p\not=q$.  If $p\not=q$, then since $\vec{F}$ is torsion free, $(p^{-1}q)^{2n+2} \not=1$, implying that $\phi$ restricted to $\vec{F}$ is a proper quotient of $\vec{F}$.  
   
    If $p = q$, since $g \not= 1$, it must be that $c_{2n}^m \not= 1$, and in particular $m > 0$.  Then 
    \begin{align*}
       1 &= \phi(p^{-1}q) = \phi(c_{2n}^m)\\
       &\overset{(\ref{rel4})}= \phi(g_{2n+1-(m-1)}c_{2n+2}^m)
    \end{align*}
   
   Note that $m < 2n+2$, so $2n+1-(m-1) > 0$.  Then $\phi(g_{2n+1-(m-1)}^{2n+4}) = \phi(c_{2n+2}^{-(2n+4)m}) = 1$, but $g_{2n+1-(m-1)} \not= 1$.  Thus once again, $\phi$ restricted to $\vec{F}$ is a proper quotient of $\vec{F}$.

    \end{proof}

\begin{proof}[Proof of Theorem \ref{thm:presentation}]Let $\phi: G \to \vec{T}$ be the surjective homomorphism that sends generators to generators.  By Lemma \ref{lem:vecFsubgroup}, $\phi$ restricted to $\vec{F}$ is an isomorphism. By Lemma \ref{lem:properquotient}, $\phi$ must then have trivial kernel.  Thus $\phi$ is an isomorphism
\end{proof}

Now we can show that only finitely many of these relations are needed, and thus give a finite presentation.

\begin{cor}\label{cor:finitepresentation}$\vec{T}$ has a finite presentation with generators $\{g_1,g_2,g_3,c_0\}$ and relations
\begin{align*}
g_1^{-1}g_3g_1 &= g_2^{-1}g_3g_2\\
g_1^{-3}g_2g_1^3 &= g_3^{-1}g_1^{-2}g_2g_1^2g_3\\
g_1^{-2}g_3g_1^2 &= g_2^{-1}g_1^{-1}g_3g_1g_2\\
g_1^{-2}g_3g_1^2 &= g_3^{-1}g_1^{-1}g_3g_1g_3\\
g_1^{-2}g_2g_1^2 &= g_2^{-1}g_1^{-1}g_2g_1g_2\\
g_1^{-2}g_2g_1^2 &= g_3^{-1}g_1^{-1}g_2g_1g_3\\
c_2g_2 &= g_1c_4\\
c_2g_3 &= g_2c_4\\
c_4g_4 &= g_3c_6\\
c_0g_1 &= c_2^3\\
c_2g_1 &= c_4^3\\
c_0^2 &= 1
\end{align*}
where the elements $c_{2n}$ and $g_n$ are defined inductively by $c_{2n} = g_{2n-1}^{-1}c_{2n-2}$ and $g_n = g_{n-2}^{-1}g_{n-1}g_{n-2}$.
\end{cor}
\begin{proof}
Thorem \ref{vecTgenerators} proves that the choice of generators is correct, so it remains to prove that the set of relations given in Lemma \ref{lem:basic relation} follow, since these are shown to be the defining relations of $\vec{T}$ in Theorem \ref{thm:presentation}. We will refer to these relations by their numbering given in Lemma \ref{lem:basic relation}.  

The first six relations imply (\ref{rel0a}), since these are the relations for $F_3$ with generators $g_1, g_2,$ and $g_3$ given in \cite{gs97}.

By definition of $c_{2n}$ in this corollary, (\ref{rel1}) is trivial.  To prove (\ref{rel2}), we do induction on both $n$ for fixed base values of $k$, and then induct on $k$.  For $k = 2$ and $n=1$, $k = 3$ and $n = 1$, and $k = 4$ and $n = 2$, the corresponding relations are given in the finite presentation.  For these fixed values of $k$, we induct on $n$ then in the following way, referring to the induction step with abbreviation ind:

\begin{align*}
    c_{2n}g_k &\overset{(\ref{rel1})}= g_{2n-1}^{-1}c_{2n-2}g_k\\
    &\overset{\text{ind}}{=} g_{2n-1}^{-1}g_1c_{2n}\\
    &\overset{(\ref{rel0a})}=g_1g_{2n+1}^{-1}c_{2n}\\
    &\overset{(\ref{rel1})}=g_1c_{2n+2}
\end{align*}

For the remaining cases, we induct on $k \geq 5$, assuming that for smaller values of $k$ and all values of $n$ that (\ref{rel2}) holds.  We will refer to using (\ref{rel2}) when $k = 2$ as the base case, or base for short, and smaller values of $k$ by ind, standing for proof by induction.

\begin{align*}
    c_{2n}g_k &\overset{\text{base}}= g_1^{-1}c_{2n-2}g_2g_k\\
    &\overset{(\ref{rel0a})}= g_1^{-1}c_{2n-2}g_{k-2}g_2\\
    &\overset{\text{ind}}= g_{1}^{-1}g_{k-3}c_{2n}g_2\\
    &= g_{1}^{-1}g_{k-3}g_1g_1^{-1}c_{2n}g_2\\
    &\overset{(\ref{rel0a})}= g_{k-1}g_1^{-1}c_{2n}g_2\\
    &\overset{\text{base}}= g_{k-1}c_{2n+2}
\end{align*}

Next we prove that (\ref{rel3}) follows by induction on $n$, noting that the base cases of induction $n = 0$ and $n = 1$ are assumed in the finite relations.  Similar to before, ind will stand for induction on $n$.

\begin{align*}
    c_{2n}g_1 &\overset{(\ref{rel2})}= g_1^{-1}c_{2n-2}g_2g_1\\
    &\overset{(\ref{rel0a})}=g_1^{-1}c_{2n-2}g_1g_4\\
    &\overset{\text{ind}}=g_1^{-1}c_{2n}^3g_4\\
    &\overset{(\ref{rel2})}=g_1^{-1}g_1c_{2n+2}^3\\
    &= c_{2n+2}^3
\end{align*}

We can now prove (\ref{rel0b}), the final set of relations.

\begin{align*}
    c_{2n}^{2n+2} &= c_{2n}c_{2n}^{2n}c_{2n}\\
    &\overset{(\ref{rel1})}=c_{2n}c_{2n}^{2n}g_{2n+1}c_{2n+2}\\
    &\overset{(\ref{rel2})}=c_{2n}g_{2n+1-2n}c_{2n+2}^{2n}c_{2n+2}\\
    &= c_{2n}g_1c_{2n+2}^{2n+1}\\
    &\overset{(\ref{rel3})}=c_{2n+2}^3c_{2n_2}^{2n+1}\\
    &=c_{2n+2}^{2n+4}
\end{align*}

Thus since $c_0^2$ is a relation, (\ref{rel0b}) follows by induction on $n$.
\end{proof}

To understand more about the structure of $\vec{T}$, we can use the observation in Lemma \ref{lem:properquotient} to show that every proper homomorphism of $\vec{T}$ factors through a certain homomorphism from $\vec{T}$ to the infinite dihedral group.

\begin{cor}
Any proper homomorphism of $\vec{T}$ factors through the homomorphism $\alpha$ from $\vec{T}$ to the infinite dihedral group, $\gen{c_0,g_1 | c_0^2 = 1, g_1c_0g_1c_0 = 1}$, where $\alpha(c_{2n}) = c_0$, $\alpha(g_{2n+1}) = g_1$, and $\alpha(g_{2n+2}) = 1$ for all $n \geq 0$.
\end{cor}
\begin{proof}
Let $\phi$ be a proper homomorphism of $\vec{T}$.  Then by Lemma \ref{lem:properquotient}, $\phi$ restricted to $\vec{F}$ is also proper, and hence the image of $\vec{F}$ under $\phi$ is abelian \cite[Theorem 4.13]{b87}.  Therefore, since $g_k^{-1}g_ng_k = g_{n+2}$ in $\vec{T}$, $\phi(g_n) = \phi(g_{n+2})$, where $n > 1$. In all the equalities that follow, we will use $\phi(g_n) = \phi(g_{n+2})$ extensively, along with all other equalities that we prove hold true in the image of $\phi$. 

We also have 
   \begin{align*}
       \phi(c_4g_4)&\overset{(\ref{rel2})}= \phi(g_1^{-1}c_2g_2g_4) = \phi(g_1^{-1}c_2g_2^2)\\
       \phi(c_4g_4)&\overset{(\ref{rel2})}= \phi(g_3c_6) \overset{(\ref{rel2})}= \phi(g_3g_1^{-2}c_2g_2^2)
   \end{align*}
      As a result, $\phi(g_3g_1^{-1}) = 1$, so $\phi(g_3) = \phi(g_1)$, extending $\phi(g_n) = \phi(g_{n+2})$ to $n = 1$.  
   
   We also have
   \begin{align*}
       \phi(c_4) &\overset{(\ref{rel1})}= \phi(g_5c_6) = \phi(g_3c_6) \overset{(\ref{rel2})}= \phi(g_3g_1^{-1}c_4g_2) = \phi(c_4g_2)
   \end{align*}
   This shows that $\phi(g_2) = 1$, and hence $\phi(g_{2n}) = 1$. 
    
    Next we show that $\phi(c_{2n+2}^2) = 1$ for any $n \geq 1$, noting that if $n=0$ we already have $c_0^2 = 1$.
    \begin{align*}
        \phi(c_{2n}g_1) &= \phi(c_{2n}g_3) \overset{(\ref{rel2})}= \phi(g_{2}c_{2n+2}) = \phi(c_{2n+2})\\
        \phi(c_{2n}g_1) &\overset{(\ref{rel3})}= \phi(c_{2n+2}^3)
    \end{align*}
    To extend this to $\phi(c_2^2)=1$, we use that $\phi(c_{2}) \overset{(\ref{rel1})}= \phi(g_{3}c_4) = \phi(g_1c_4)$, hence $\phi(g_1) = \phi(c_2c_4^{-1}) = \phi(c_2c_4)$. Now $\phi(c_2g_1) =\phi(c_2g_3) \overset{(\ref{rel2})}= \phi(g_2c_4) = \phi(c_4)$, so $\phi(g_1) = \phi(c_2^{-1}c_4)$.  Combining these gives $\phi(c_2c_4) = \phi(c_2^{-1}c_4)$, so $\phi(c_2^2) = 1$ as well, showing that $\phi(c_{2n}^2) = 1$ for any $n \geq 0$.  
    
    Thus $\phi(g_1)$ and $\phi(c_0)$ generate the image of $\phi$, and we can use the relations shown to hold in $\phi(\vec{T})$ and the finite presentation given in Corollary \ref{cor:finitepresentation} to verify that the relations in the infinite dihedral group hold, proving the corollary. We will call the generators $\phi(g_1)$ and $\phi(c_0)$ as simply $g_1$ and $c_0$.  The first six relations in the finite presentation are all trivial.  Since $\phi(c_2^2) = 1$, the relation $c_0g_1 = c_2^3$ becomes $c_0g_1 = c_2$.  This follows naturally from the definition of $c_2 = g_1^{-1}c_0$ by inverting both sides, since $c_0$ and $c_2$ are their own inverses.  Likewise, since $\phi(g_2) = 1$, $c_2g_3 = g_2c_4$ becomes $c_4 = c_2g_1 = c_0g_1g_1$, which follows from the definition $c_4 = g_3^{-1}c_2 = g_1^{-1}c_2$ in the same way.  The relation $c_2g_2 = g_1c_4$ then becomes $c_0g_1 = g_1c_0g_1g_1$, which simplifies to $c_0 = g_1c_0g_1$, which can be rewritten as one of the two relations in the presentation of the infinite dihedral group: $g_1c_0g_1c_0 = 1$.  The relation $c_0^2 = 1$ remains unchanged and is the other relation in the presentation of the infinite dihedral group.  The remaining relations are easily verified to be unnecessary.
\end{proof}

\section{The Stallings 2-core of a subgroup of \texorpdfstring{$T$}{T} and an annular diagram group presentation for \texorpdfstring{$\vec{T}$}{Jones' subgroup of T}} \label{stallingscore}
\subsection{The Stallings 2-core}
In his original paper \cite{stallings}, Stallings introduced a notion of foldings of a graph, which gave an algorithm for solving the membership problem for finitely generated subgroups of free groups.  Specifically, given a set of generators of a free group, each can be drawn as a loop with edges labeled by the generators, and then adjacent edges with identical labels can be ``folded together" (i.e., identified with each other).  Eventually this process will finish and reveal the free generators of the subgroup.  Then any element of the subgroup can be read in this graph called a Stallings core, and in this case the core is said to accept the element.

In \cite{gs97}, Guba and Sapir extended the notion of Stallings foldings of edges to foldings of cells and created the notion of a Stallings 2-core corresponding to a finitely generated subgroup of a diagram group.  In general, solving the membership problem is much more difficult in this class of groups, and sometimes even impossible, but what they showed is that the constructed 2-core accepts at least the subgroup, and possibly more.  Thus the Stallings 2-core provides one way to prove that a subgroup of a diagram group is a strict subgroup, by simply checking whether the 2-core accepts the generators of the group.

For example, $F$ is a diagram group \cite{gs97}, and hence each subgroup $H$ has a corresponding 2-core. Let the subset of $F$ accepted the 2-core of $H$ be denoted $C(H)$.  Then $C(H)$ is a subgroup, and if $f \in C(H)$ and $f$ fixes a dyadic rational number $\alpha \in (0,1)$, then the following functions are called components of $f$:
\[ f_1(t) = 
    \begin{cases} 
      f(t) & t \in[0,\alpha] \\
      t & t \in (\alpha,1] \\
   \end{cases}
   \hspace{0.3in} f_2(t) =
   \begin{cases}
      t & t \in[0,\alpha] \\
      f(t) & t \in (\alpha,1]
   \end{cases}.
\]
Then the core must then accept $f_1$ and $f_2$, hence $C(H)$ must be closed under adding components in this way.  However, it was proved in \cite{g16} that in the case of $F$, $C(H)$ is exactly the smallest subgroup containing $H$ and closed under adding components of elements of $H$, thus classifying the elements accepted by the core completely. In this section, we extend the notion of a 2-core to $T$ by using the commonalities between pair of trees representations of elements of $F$ and $T$.  

Given an element of $T$ with a pair of trees representation $(R, S, n)$, direct all the edges in $R$ and $S$ away from the root of the corresponding tree.  In this way, there is a unique directed path from the root to any vertex in each tree.  Moreover, once the leaves of the two trees are identified, any leaf has exactly two directed paths, called the positive and negative paths corresponding to the input and output trees respectively.  This directed graph associated to an element of $T$ will be called the element's diagram, and by labeling left edges with $0$ and right edges with $1$, it will still be possible to distinguish between left and right edges in the diagram.  The diagram is called reduced if it comes from a reduced pair of trees representation.

\begin{defn}
The core of a finitely generated subgroup $H = \gen{x_1,\ldots,x_n} \leq T$ is a labeled directed graph constructed in the following way. First, begin with all the reduced diagrams for $x_1,\ldots,x_n$ and identify all the roots of trees in these diagrams together.  Proceed by doing the following steps as many times as possible:
\begin{enumerate}
    \item If two vertices are identified, identify their left children and left edges, along with their right children and right edges.
    \item If two vertices have their left children and their right children identified respectively, then identify the vertices.
\end{enumerate}
This process will stop since there are only finitely many edges and vertices.
\end{defn}

Now, an element $x \in T$ is said to be accepted by a core if there is a homomorphism $\phi$ of labeled directed graphs from the reduced diagram of $x$ to the core, where a graph homomorphism is a map such that vertices are sent to vertices, edges are sent to edges of the same label, and if $(u,v)$ are the vertices of a directed edge, then so are $(\phi(u),\phi(v))$.  Checking whether such a homomorphism exists is done by first mapping the roots of the reduced diagram of $x$ to the root of the core, and then identifying edges and vertices from the diagram of $x$ with edges and vertices of the core using the same two steps outlined above repeatedly.  Either a natural graph homomorphism to the core will arise from the identifications, or a contradiction (e.g., a vertex of in the diagram of $x$ being identified with two different vertices in the core, or a vertex in $x$ which has children is identified with a vertex in the core that has no children) is reached and the process cannot continue, implying that no such homomorphism exists.  An example of computing the core of $F = \gen{x_0,x_1}$ is given in Figure \ref{excore}.

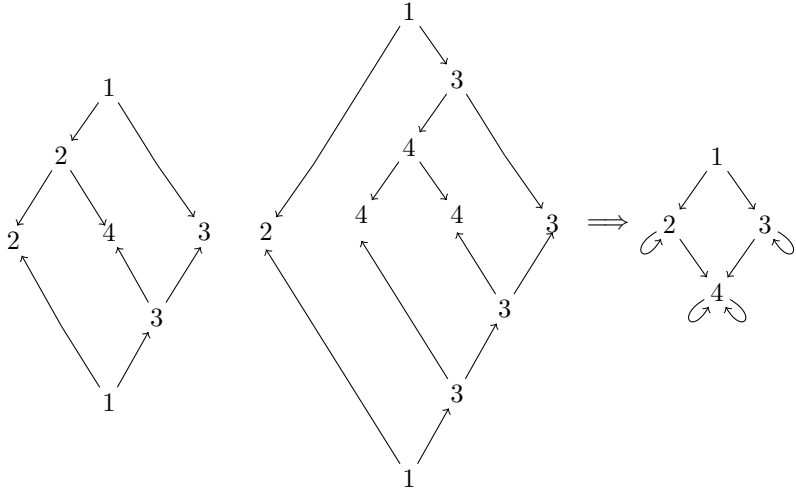
\begin{figure}[ht]\centering
\begin{tikzpicture}  [
    grow                    = down,
    level distance          = 0.4in,
    sibling distance        = 0.5in,
    anchor                  = base,
    baseline                = 0.4in,
    ]
    \node [dummy] (root) {1}
    child[->] { node [dummy] {2}
      child[->] {node [dummy] {}}
      child[->] {node [dummy] {}} 
     }
    child { 
      child {edge from parent [draw=none]}
      child[->] {node [dummy] {3}}
    };
    \node [dummy,below=1.57in](broot) at (root) {1} [grow'=up]
    child {
      child[->] {node [dummy] {2}}
      child {edge from parent [draw=none]}
    }
    child[->] { node [dummy] {3}
      child[->] {node [dummy] {4}}
      child[->] {node [dummy] {}}
    };
\end{tikzpicture}
\hspace{0.1in}
\begin{tikzpicture}  [
    grow                    = down,
    level distance          = 0.4in,
    sibling distance        = 0.5in,
    anchor                  = base,
    baseline                = 0in,
    ]
    \node [dummy] (root) {1}
      child {
        child {
          child[->] { node [dummy] {2}}
          child {edge from parent [draw=none]}
        }
        child[->] {node [dummy]{} edge from parent [draw=none]}
      }
      child[->] {node [dummy] {3}
        child[->] {node [dummy] {4}
          child[->] {node [dummy] {4}}
          child[->] {node [dummy] {4}}
         }
        child[-] {
          child{edge from parent [draw=none]}
          child[->] {node [dummy] {3}}
        }
      };
    \node [dummy,below=2.37in](broot) at (root) {1} [grow'=up]
    child {
      child {
        child[->] {}
        child {edge from parent [draw=none]}
      }
      child {edge from parent [draw=none]}
    }
    child[->] { node [dummy] {3}    
      child[-] {
        child[->] 
        child {edge from parent [draw=none]}
      }
      child[->] { node [dummy] {3}
          child[->] 
          child[->]
      }
    };
\end{tikzpicture}
\begin{tikzpicture} [
    grow                    = down,
    level distance          = 0.4in,
    sibling distance        = 0.5in,
    anchor                  = base,
    baseline                = 0.76in,
    ]
    \node [dummy] {1}
     child[->]{node(2) [dummy] {2}
         child{edge from parent [draw=none]}
         child[->]{node [dummy] {\phantom{4}}}
     }
     child[->]{node(3) [dummy] {3} 
             child[->]{node [dummy](4) {4}}
             child{edge from parent [draw=none]}
     };
     \draw[->] (4) to [out=330,in=300,looseness=10] (4);
     \draw[->] (4) to [out=210,in=240,looseness=10] (4);
     \node [dummy,left=1em] at (2) {$\implies$};
     \draw[->] (2) to [out=210,in=240,looseness=10] (2);
     \draw[->] (3) to [out=330,in=300,looseness=10] (3);
\end{tikzpicture}
\caption{Depicted are the identified pairs of trees for $x_0$ on the left and $x_1$ in the middle.  To create the core, first identify all roots, where identifications are labeled by numbers. Thus all roots are first labeled by 1.  Next identify all left children of vertex 1 as vertex 2, and identify all right children of vertex 1 as vertex 3.  Notice that in one instance, vertex $2$ has left child $2$ and vertex $3$ has right child $3$, so update all corresponding children of vertices labeled $2$ and $3$ accordingly.  Then the right child of vertex $2$ needs a label, so label it vertex $4$.  Then $4$ is also in one instance the left child of $3$, so update the core once again.  Finally, all possible identifications and labels have been completed, so the core is complete, and it has $4$ vertices and is shown on the right.}\label{excore}
\end{figure}

Let the subset of $T$ accepted by the core of a finitely generated subgroup $H$ be denoted by $C(H)$.  Then $C(H)$ is also a subgroup of $T$ that contains $H$, although it may be larger.

\begin{lem} 
$C(H)$ is a subgroup that contains $H$.
\end{lem}
\begin{proof}
By construction, $C(H)$ contains the generators of $H$, so it suffices to prove that $C(H)$ is a subgroup.  
Suppose $x\in C(H)$.  It is clear that $x^{-1} \in C(H)$ since the diagrams for $x$ and $x^{-1}$ are isomorphic.

Suppose $x, y \in C(H)$.  Then the unreduced tree diagram corresponding to $xy$ is accepted by the core, since the only overlap between the two diagrams in the unreduced product diagram is the root, which is sent to the same place in the core.  Furthermore, if a dipole is removed from a diagram accepted by the core, then the vertices which are identified by removing the dipole were also identified in the core, so the core accepts the reduced diagram.  Thus $xy \in C(H)$.
\end{proof}

The most important aspect of the core of $H$ for this paper is that any element accepted by the core of $H$ is piecewise-dyadic $H$, and this is a direct corollary of the following two lemmas.  A function $f \in T$ is piecewise-dyadic $H$ if  
\[f(t) = \begin{cases} f_1(t) & t \in [0,\alpha_1] \\
    f_2(t) & t \in (\alpha_1,\alpha_2] \\
    \phantom{f_n} \vdots & \\
    f_n(t) & t \in (\alpha_{n-1},1] \\
    \end{cases}\]
where each $f_i$ is in $H$ and $\alpha_i$ is a dyadic rational.  Thus $f$, although it may not be in $H$ itself, is composed of pieces of functions from $H$.  For example, it is an observation that if $f\oplus g$ is in $C(H)$, then so is $1 \oplus g$, although the group generated by $f \oplus g$ need not contain $1 \oplus g$.

Golan proved this fact for the the case of $F$ \cite{g16}, but it is also true for $T$ with very little alteration.  The proofs presented here for Lemmas \ref{golanlem1} and \ref{golanlem2} are the result of discussion of the first author with Golan.

Since edges on trees and the core are labeled with $0$ for left and $1$ for right, paths and their labels will be referred to interchangeably.  For every leaf in a pair of trees diagram of $x \in T$, let $t_+, t_- \in \{0,1\}^*$ be the labels of the unique positive and negative paths to the leaf. Thus $x$ is associated with $t_+ \to t_-$, which has the following meaning.  Viewing $x$ as a function on binary fractions, $t_+\to t_-$ means that $x$ applied to a binary fraction beginning with $t_+$ replaces the prefix with $t_-$.

\begin{lem}\label{golanlem1} Let $H = \gen{x_1,\ldots,x_n} \leq T$, and let $u, v$ be a vertices of the diagrams for some $x_i, x_j$ respectively. If $u$ and $v$ are identified in the core of $H$, then there exists a non-negative integer $l_{u,v}$ such that for any directed path with label $t_u$ to $u$ in $x_i$ and any directed path with label $t_v$ to $v$ in $x_j$, as well as for any finite binary word $\alpha$ with $\abs{\alpha} \geq l_{u,v}$, there exists a not necessarily reduced diagram for some $g \in H$ containing $t_u\alpha \to t_v\alpha$.
\end{lem}
\begin{proof}
We proceed by considering the construction of the core.  At the beginning, only roots are identified, so the only paths to the roots have empty label, and hence for $l = 0$ and any $\alpha$ we can simply take the unreduced trivial diagram of depth $\abs{\alpha}$.

Now suppose that the result holds up to a certain point in the construction of the core, and then two vertices are identified in the core.  Since there are two ways to identify vertices, we have two cases to consider.

First, assume that $u$ and $v$ are the left or right children of two identified vertices, $w_u$ and $w_v$ from $x_i$ and $x_j$ respectively.  Since the cases are symmetric, assume they are both left children.  If neither $u$ nor $v$ are leaves, then they have unique paths in $x$ and $y$ that go through $w_u$ and $w_v$.  So there exists $l' = l_{w_u,w_v}$ such that if $\abs{\alpha} \geq l'$, then there exists $g \in H$ containing $t_{w_u}\alpha \to t_{w_v}\alpha$.  Thus if $\abs{\alpha} \geq l' - 1$, there exists $g \in H$ containing $t_u\alpha \to t_v\alpha$ since $t_u = t_{w_u}0$, and $t_v = t_{w_v}0$. Thus $l_{u,v} = \max\{l'-1,0\}$ works.

Suppose that $u$ is a leaf.  Then up to taking $x_i^{-1}$ instead of $x_i$, we may suppose that $w_u$ is in the negative part of the diagram of $x_i$, hence $x_i$ contains $t_{u,+} \to t_{w_u}0$ where $t_{u,+}$ is the unique positive path to $u$.  Using the same idea as before, for any $\abs{\alpha} \geq l_{w_u,w_v} - 1$, we can find a $g \in H$ such that $g$ contains $t_{w_u}0\alpha \to t_{w_v}0\alpha = t_v\alpha$.  Observe that $x_ig$ in an unreduced form contains $t_{u,+}\alpha \to t_{w_u}0\alpha$, hence we can choose $l_{u,v} = \max\{l_{w_u,w_v}-1,0\}$.  If $v$ or both $u$ and $v$ are leaves, the process is very similar.

For the second case, suppose that the children of $u$ and $v$ have already been identified, and hence $u$ and $v$ are being identified.  This time, as $u$ and $v$ have children, they are clearly not leaves.  Let $u_l, u_r, v_l, v_r$ be the left and right children of $u$ and $v$.  Then choose $l = \max\{l_{u_l,v_l},l_{u_r,v_r}\} + 1$.  Then if $\abs{\alpha} \geq l$, we can write $\alpha$ as $0\alpha'$ or $1\alpha'$.  Since both cases are similar, we will consider the former case.  Then $\abs{\alpha'} \geq l_{u_l,v_l}$, hence there exists $g \in H$ taking $t_{u_l}\alpha' \to t_{v_l}\alpha'$.  Since $t_{u_l}\alpha' = t_u\alpha$ and $t_{v_l}\alpha' = t_v\alpha$, we are done.
\end{proof}

This lemma about paths to vertices in diagrams of $H$ identified in the core generalizes to any two paths in the core with the same end point.

\begin{lem}\label{golanlem2}
Let $H = \gen{x_1,\ldots,x_n} \leq T$, and let $p_1$ and $p_2$ be the labels of directed paths in $C(H)$ from the root to the same vertex.  Then there exists a non-negative integer $l$ such that for every $\alpha \in \{0,1\}^*$ with $\abs{\alpha} \geq l$, there exists $g \in H$ containing $p_1\alpha \to p_2\alpha$.
\end{lem}
\begin{proof}
Observe that any path $p$ in the core can be decomposed into $p_1,\ldots,p_k$ where each $p_i$ is the label of some path (not necessarily from the root) in some $x_j$, and the beginning vertex of $p_{i+1}$ has been identified with the end vertex of $p_i$, and $p_1$ begins at the root and $p_k$ ends at the final vertex of $p$.  

Thus let $p_1 = p_1,\ldots,p_k$ and $p_2 = q_1,\ldots,q_l$ be such decompositions.  We proceed by induction on $k+l$, with the base case of $k+l = 2$ being the previous proposition.  Now suppose the proposition holds if $m = k+l$, and up to taking the inverse of the element $g$ that we find, we may assume that $k \geq 2$.  Let $p_1 \in x_i$ and $p_2 \in x_j$, with $u$ the last vertex of $p_1$ and $v$ the first vertex of $p_2$.  Since $v$ is on a directed path in $x_j$, it is not a leaf, so there exists a unique path $p_1'$ in $x_j$ to $v$.  Now, since $p_1$ and $p_1'$ satisfy the conditions of the previous proposition, there exists $l'$ such that for any $\abs{\alpha'} \geq l$, there exists $g' \in H$ containing $t_u\alpha' \to t_v \alpha'$.  Thus if $\abs{\alpha} \geq l - \abs{p_2\ldots p_k}$, then $g'$ contains $t_up_2\ldots p_k \alpha \to t_vp_2\ldots p_k\alpha$.  Since $p_1'$ and $p_2$ are in the same tree, by induction there is a $g \in H$ containing $t_vp_2\ldots p_k\alpha \to q_1\ldots q_l \alpha$ if $\abs{\alpha} \geq l''$ for some $l''$.  Up to originally taking a larger quantity, we may assume $l' \geq l''$, and hence the element $g'g$ contains $p_1\alpha \to p_2 \alpha$.
\end{proof}

In particular, this last proposition directly implies that every element of $C(H)$ is piecewise-dyadic $H$.  Notice that this implies that if $H \leq F$, then any element in $C(H)$ is piecewise-dyadic $H$, and since every element of $H$ stabilizes $0$, so does every element in $C(H)$.  Thus $C(H) \leq F$ and this extension of the Stallings 2-core to $T$ coincides with its definition in $F$.

\subsection{\texorpdfstring{$\vec{T}$}{Jones' subgroup of T} as an annular diagram group}

\begin{prop}
If $H < T$ is such that $H = C(H)$, then $H$ is an annular diagram group.
\end{prop}
\begin{proof}
The core of $H$ consists of carets with labels, which can be used directly to form the presentation for the annular diagram group.  Let $e$ be the name of the vertex in the core which was identified with all the roots of the generators of $H$.  Suppose a caret in the core has label $x$ on the root, $y$ on the left child, and $z$ on the right child.  Then the corresponding rewriting is $x = yz$.  Let $S$ be the set of all distinct vertices in the core, and $R$ be the set of all rewriting rules for every caret in the core.  Then we claim that the collection of all annular $e-e$ diagrams over the presentation $\gen{S | R}$ is an annular diagram group $A$ isomorphic to $H$.  

Indeed, let $\phi: A \to H$ be a homomorphism described as follows.  For $\Delta \in A$ a reduced annular diagram, every cell in $\Delta$ has exactly one top edge and two bottom edges, and thus each cell can be identified with a caret as in Figure \ref{adiagramc}, which describes a similar situation for going between annular diagrams and tree diagrams for elements of $T$.  Since each vertex in the core of $H$ has at most two children, there are no two distinct rewriting rules $x = yz$ and $x = uv$.  Moreover, since $\Delta$ is reduced, then since each unique pair of left and right children in the core of $H$ have at most one parent, there are no dipoles created in this identification.  As a result, replacing the cells of $\Delta$ with carets as described results in a tree diagram, in exactly the same way as annular diagrams in $\mathcal{D}^a(\gen{x | x = xx})$ correspond to tree diagrams.  Call this tree diagram $(R,S,n)$, and define $\phi(\Delta) = (R,S,n)$.  Since every reduced representation of an element of $H$ has a unique identification with the core of $H$ consisting of carets labeled by the rewriting rules of $A$, it is clear that $\phi$ is surjective.  Likewise, since no dipoles are created, no non-trivial diagram is sent to the identity, hence $\phi$ is injective. Finally, $\phi$ is a homomorphism since reduction of diagrams in $A$ corresponds to removing dipoles in $H$, thus multiplication and reduction of two elements in $A$ gives the same element of $H$.
\end{proof}

\begin{prop}
$C(\vec{T}) = \vec{T}$, and in particular $\vec{T} = \mathcal{D}^a(\gen{e,f | e = ff, f = fe},e)$.\end{prop}
\begin{proof}
Let $f \in C(\vec{T})$.  Then $f$ is piecewise-dyadic $\vec{T}$, and can be written as 
\[f(t) = \begin{cases} f_1(t) & t \in [0,\alpha_1] \\
    f_2(t) & t \in (\alpha_1,\alpha_2] \\
    \phantom{f_n} \vdots & \\
    f_n(t) & t \in (\alpha_{n-1},1] \\
    \end{cases}\]

Now, each $f_i(t)$ either preserves or switches dyadic parity.  Since $f$ is continuous, $f_i(\alpha_i) = f_{i+1}(\alpha_i)$, where $f_{n+1} := f_1$.  Thus if $f_i$ preserves the dyadic parity of $\alpha_i$, so does $f_{i+1}$.  By induction, all $f_i$ do the same, i.e., all preserve dyadic parity or all switch it.  Thus $f$ does the same, and hence $f \in \vec{T}$ by Theorem \ref{dyadicparity}.

The particular presentation for $\vec{T}$ as an annular diagram group is a simple computation of the core of $\vec{T}$.
\end{proof}

Notice that the presentation in the annular diagram group definition for $\vec{T}$ is a simple Tietze transformation away from the presentation in the annular diagram group definition of $T_3 = \mathcal{D}^a(\gen{f | f = f^3},f)$.  However, although $\vec{F}$ is isomorphic to $F_3$, this is not the case for $\vec{T}$.

\begin{prop}
$T_3$ does not contain any element of order two, hence it is not isomorphic to $\vec{T}$.
\end{prop}
\begin{proof}
First, note that elements of $T_3$ have pair of trees representations exactly like elements of $T$, except that the trees are ternary rather than binary, just as the functions have slopes integer powers of $3$ and breakpoints at 3-adic rationals.  

Now suppose that $f \in T_3$ has order $2$.  Let $f(0) = \alpha$.  Then since $f^2 = 1$, $f^2(0) = f(\alpha) = 0$. Since $f$ is continuous then, $f([0,\alpha]) = [\alpha,1]$ and $f([\alpha,1]) = [0,\alpha]$.

Now, let $(R,S,k)$ be a pair of trees representation of $f$ with each tree containing $n$ leaves.  Then since $f(0) = \alpha$, $f$ sends the first leaf of $R$ to a leaf of $S$ corresponding to some interval that begins with $\alpha$.  Likewise, since $f(\alpha) = 0$, $f$ sends some vertex in $R$ whose interval begins with $\alpha$ to the first leaf of $S$.  In particular, both $R$ and $S$ contain leaves whose intervals begin with $\alpha$.  Let $R_{-,\alpha}$ and $S_{-,\alpha}$ be the leaves of $R$ and $S$ respectively whose intervals combine to give $[0,\alpha]$.  Similarly, let $R_{+,\alpha}$ and $S_{+,\alpha}$ be the remaining leaves in each tree.

There is a smallest full ternary tree that contains such a leaf whose interval begins with $\alpha$, which is a subtree of $R$ and $S$.  Since adding vertices to the tree consists of giving one of the leaves three children, there is a net gain of $2$ leaves, so the parity of the number of leaves of the tree doesn't change.  Moreover, the intervals associated with the three children of a vertex partition the interval associated with that vertex, hence the parity of the number of leaves whose intervals partition $[0,\alpha]$ does not change.  In other words, $|R_{-,\alpha}|$ and $|S_{-,\alpha}|$ have the same parity. But since the leaves in $R_{-,\alpha}$ are identified with the leaves in $S_{+,\alpha}$, $|R_{-,\alpha}|$ and $|S_{+,\alpha}|$ have the same parity.  Thus $|S_{-,\alpha}|+|S_{-,\alpha}|$ is even.  But by a similar argument, every full ternary tree has an odd number of leaves, since the smallest full ternary tree has $1$ leaf, the root, and adding three children to a leaf changes the number of leaves by $2$.  Thus $|S_{-,\alpha}|+|S_{-,\alpha}|$ is odd, a contradiction.
\end{proof}


\begin{thebibliography}{9}


\bibitem{b87} Kenneth S. Brown, Finiteness properties of groups. Proceedings of the Northwestern conference on cohomology of groups (Evanston, Ill., 1985). J. Pure Appl. Algebra 44 (1987), no. 1-3, 45–75. 
\bibitem{cfp96} J. W. Cannon, W. J. Floyd, W. R. Parry, Introductory notes on Richard Thompson's groups. \emph{Enseign. Math. (2)} 42 (1996), no. 3-4, 215–256.
\bibitem{gs15a} Gili Golan, Mark Sapir, On Jones' subgroup of R. Thompson group F. \emph{J. Algebra} 470 (2017), 122–159.
\bibitem{gs15b} Gili Golan, Mark Sapir, On Subgroups of R. Thompson's Group F. \emph{Trans. Amer. Math. Soc.}  (2017). {\color{blue}\url{https://arxiv.org/abs/1508.00493}}.
\bibitem{g16} Gili Golan, The generation problem in Thompson group F. {\color{blue}\url{https://arxiv.org/abs/1608.02572v1}}.
\bibitem{gs97} Victor Guba, Mark Sapir, Diagram groups. \emph{Mem. Amer. Math. Soc.} 130 (1997), no. 620, viii+117 pp.
\bibitem{j83}  Vaughan Jones, Index for subfactors. \emph{Invent. Math.} 72 (1983), no. 1, 1–25.
\bibitem{j14} Vaughan Jones, Some unitary representations of Thompson's groups F and T. \emph{J. Comb. Algebra} 1 (2017), no. 1, 1–44.
\bibitem{j16} Vaughan Jones, A no-go theorem for the continuum limit of a periodic quantum spin chain.
{\color{blue}\url{https://arxiv.org/abs/1607.08769}}
\bibitem{JonPA} Vaughan Jones, Planar algebras, {I}, \emph{New Zealand J. Math.} (1998).	
\bibitem{stallings}John R. Stallings, Topology of finite graphs. \emph{Invent. Math.} 71 (1983), no. 3, 551–565.

\end{thebibliography}
\end{document}